\documentclass[11pt,reqno]{amsart}
\usepackage[margin=1in]{geometry}
\usepackage[tt=false]{libertine}
\usepackage{caption}

\usepackage{verbatim}

\usepackage[dvipsnames]{xcolor}
\makeatletter
\def\@tocline#1#2#3#4#5#6#7{\relax
  \ifnum #1>\c@tocdepth 
  \else
    \par \addpenalty\@secpenalty\addvspace{#2}%
    \begingroup \hyphenpenalty\@M
    \@ifempty{#4}{%
      \@tempdima\csname r@tocindent\number#1\endcsname\relax
    }{%
      \@tempdima#4\relax
    }%
    \parindent\z@ \leftskip#3\relax \advance\leftskip\@tempdima\relax
    \rightskip\@pnumwidth plus4em \parfillskip-\@pnumwidth
    #5\leavevmode\hskip-\@tempdima
      \ifcase #1
       \or\or \hskip 1em \or \hskip 2em \else \hskip 3em \fi%
      #6\nobreak\relax
    \hfill\hbox to\@pnumwidth{\@tocpagenum{#7}}\par
    \nobreak
    \endgroup
  \fi}
\makeatother

\usepackage{hyperref}

\hypersetup{
	colorlinks=true,
	pdfpagemode=UseNone,
    citecolor=OliveGreen,
    linkcolor=NavyBlue,
    urlcolor=Magenta,
	pdfstartview=FitW
}
\usepackage{mathabx}
\usepackage[capitalize,nameinlink]{cleveref}

\usepackage{graphicx}
\usepackage{enumerate}
\usepackage{bm}
\usepackage{tikz}

\makeatletter
\def\paragraph{\@startsection{paragraph}{4}%
  \z@\z@{-\fontdimen2\font}%
  {\normalfont\bfseries}}
\makeatother

\usepackage{bookmark}
\usepackage{amsthm}
\usepackage{amssymb}
\usepackage{amsmath}
\usepackage{amsfonts}
\usepackage{cancel}
\usepackage{times}
\usepackage{bbm}
\usepackage{cite}
\usepackage{color, soul}
\usepackage{wasysym}

\def\qed{\hfill $\vcenter{\hrule height .3mm
\hbox {\vrule width .3mm height 2.1mm \kern 2mm \vrule width .3mm
height 2.1mm} \hrule height .3mm}$ \bigskip}

\def \Sph{\mathbb{S}^{d-1}}
\def \RR {\mathbb R}
\def \NN {\mathbb N}
\def \EE {\mathbb E}

\def \PP {\mathbb P}
\def \eps {\varepsilon}

\newtheorem{theorem}{Theorem}
\newtheorem{lemma}[theorem]{Lemma}

\newtheorem{proposition}[theorem]{Proposition}
\newtheorem{corollary}[theorem]{Corollary}
\theoremstyle{definition}
\newtheorem{definition}[theorem]{Definition}
\theoremstyle{remark}

\long\def\symbolfootnotetext[#1]#2{\begingroup
\def\thefootnote{\fnsymbol{footnote}}\footnotetext[#1]{#2}\endgroup}

\newcommand{\Q}{Q_q}
\newcommand{\qd}{\delta_q}
\newcommand{\qha}{\delta_{\mathrm{Haus}}}
\newcommand{\C}{W}
\newcommand{\dpdelta}{\delta}

\title[Archimedes Meets Privacy]{Archimedes Meets Privacy: \\On Privately Estimating Quantiles in High Dimensions\\ Under Minimal Assumptions}

\author[Ben-Eliezer]{Omri Ben-Eliezer}
\address[Ben-Eliezer]{Department of Mathematics\\
MIT}
\email{omrib@mit.edu}
\thanks{}

\author[Mikulincer]{Dan Mikulincer}
\address[Mikulincer]{Department of Mathematics\\
MIT}\email{danmiku@mit.edu}
\thanks{DM is supported by a Vannevar Bush Faculty Fellowship ONR-N00014-20-1-2826}

\author[Zadik]{Ilias Zadik}
\address[Zadik]{Department of Mathematics\\
MIT}
\email{izadik@mit.edu}
\thanks{IZ is supported by the Simons-NSF grant DMS-2031883 on the Theoretical Foundations of Deep Learning and the Vannevar Bush Faculty Fellowship ONR-N00014-20-1-2826.}

\begin{document}
	
	\maketitle
	
	\begin{abstract}
	The last few years have seen a surge of work on high dimensional statistics under privacy constraints, mostly following two main lines of work: the ``worst case'' line, which does not make any distributional assumptions on the input data; and the ``strong assumptions'' line, which assumes that the data is generated from specific families, e.g., subgaussian distributions.
    In this work we take a middle ground, obtaining new differentially private algorithms with polynomial sample complexity for estimating quantiles in high-dimensions, as well as estimating and sampling points of high Tukey depth, all working under very mild distributional assumptions. 
    
    From the technical perspective, our work relies upon deep robustness results in the convex geometry literature, demonstrating how such results can be used in a private context. 
    Our main object of interest is the (convex) floating body (FB), a notion going back to Archimedes, which is a robust and well studied high-dimensional analogue of the interquantile range.  We show how one can privately, and with polynomially many samples, (a) output an approximate interior point of the FB -- e.g., ``a typical user'' in a high-dimensional database -- by leveraging the robustness of the Steiner point of the FB; and at the expense of polynomially many more samples, (b) produce an approximate uniform sample from the FB, by constructing a private noisy projection oracle.
	\end{abstract}

\tableofcontents

\section{Introduction}
Computing statistics of large, complex high-dimensional datasets under privacy constraints is a fundamental challenge in modern data science. In this work, we study \emph{the sample complexity} of several different tasks related to estimating the quantiles of a $d$-dimensional distribution $\mathcal{D}$, from having sample access to it, under privacy constraints. Our first focus is on estimating quantiles of fixed marginals. For a random vector $X \in \RR^d$, the quantiles of $X$ along direction $\theta$ are the quantiles of the (one-dimensional) marginal $\langle X, \theta \rangle$, which we denote by $\Q(\langle X, \theta\rangle)$ for $q \in [0,1]$. Our second focus is on estimating a convex region called the floating body of a $d$-dimensional distribution, an analogue of the $1$-dimensional interquantile range exhibiting rich mathematical properties, investigated over decades of research in convex geometry and sharing connections with the work of Archimedes  (see, e.g., the survey \cite{nagy2019halfspace}). Formally, the $q$-floating body of a random vector $X\in \RR^d$ is given by
\begin{align}\label{eq:float_dfn} 
F_q(X)= \bigcap\limits_{\theta\in \Sph} \left\{x \in \RR^d : \langle x, \theta\rangle \leq \Q(\langle X, \theta\rangle)\right\},
\end{align}and in statistical language corresponds to the points of so-called \emph{Tukey-depth} at least $1-q$ \cite{nagy2019halfspace}. When clear from context, we also denote by $F_q(\mathcal{D})$ the floating body of the random vector $X \sim \mathcal{D}.$

Similar to their $1$-dimensional analogues, the quantiles in high-dimensions are preferred to other descriptive statistics and location estimators due to their robustness properties \cite{Huber81,hampel2011robust} and high breakdown points \cite{Aloupis03,davies2007breakdown,TangP16}. As opposed to other statistics, marginal quantiles can be estimated under the mere minimal distributional assumption of having non-trivial mass around the location of the quantiles (see e.g. \cite[Lemma 3]{avella2019differentially}). Moreover, the floating body can also be estimated in the Hausdorff distance for log-concave measures and stable laws \cite{anderson2020efficiency}. 

Besides statistical guarantees, an essential property of an estimator in modern data science is to be private. For instance, many large healthcare datasets contain sensitive patient information \cite{DankarE13,FredriksonLJLPR14,YukselKO17,ZhangLZHS17} and it is crucial for an estimator to balance the trade-off between protecting the privacy of the input while producing accurate statistical results.

Differential privacy (DP) is the leading notion in quantifying privacy guarantees of a randomized algorithm over an input dataset. In this work, we focus on algorithms satisfying \emph{pure} DP guarantees. Given a parameter $\eps>0$, $\eps$-differential privacy guarantees that the probability of a given output cannot change more than a multiplicative factor $e^{\varepsilon}$ after the arbitrary change of the input data of one data item (e.g., a single user record in a large database). More generally, $\eps$-DP can be defined as follows with respect to the Hamming distance $d_H\left(X, X'\right):= |\{i \in [n] \Big| X_i \neq
X_i'\}|,$ defined for two $n$-tuples of samples $X=(X_1,\ldots,X_n),X'=(X'_1,\ldots,X'_n) \in (\mathbb{R}^d)^n$.
\begin{definition}[$\eps$-DP]\label{definition:epsilon-privacy}
 A randomized algorithm $\mathcal{A}$ is $\eps$-differentially private if for all subsets $S \in \mathcal{F}$ of the output space $(\Omega, \mathcal{F})$ and $n$-tuples of samples $X_1,X_2 \in (\mathbb{R}^d)^n$, it holds that \begin{equation} \label{privdfn}\mathbb{P}\left(\mathcal{A}(X_1) \in S\right) \leq e^{\varepsilon d_H\left(X_1,X_2\right)}\mathbb{P}\left(\mathcal{A}(X_2) \in S\right).\end{equation}
\end{definition}
While there has been a lot of work on differential privacy, it usually analyzes algorithms under worst-case assumptions. The most closely related to our work is the line of research on privately outputting an interior point of the convex hull of $n$ input points in high dimensions for a worst-case input \cite{Stemmer19Centerpoint,KaplanHalfspaces2020,Sharir2020ConvexHull,SadigurschiRectangle2021}. Only recently there has been a lot of statistical work analyzing differentially private estimators under distributional assumptions for various statistical tasks e.g.~\cite{BorgsNeurIPS15, DuchiWJ16, borgs2018revealing, Mixtures,Kamath0SU19,SU19, MRF, DiscreteGaussian,  tzamos2020optimal, HeavyTailed}. On the topic of the present work, a few years ago \cite{tzamos2020optimal} established the exact trade-off between privacy and accuracy of median estimation in one-dimension. One of the main motivations of the present work is to make a first step towards understanding the high-dimensional counterparts of these privacy-accuracy tradeoffs, by proposing the floating body as a natural estimate under privacy constraints in high dimensions. In particular, the convexity of the floating body allows us to invoke powerful robustness theorems, originating from decades of research in convex geometry. It is this combination, of classical theorems with the modern framework of differential privacy, which yields several new and non-trivial results. 

\subsection{Our Contributions}
\paragraph{The minimal assumptions.}
It is known that even in one dimension, private quantile estimation is impossible for arbitrary measures \cite{tzamos2020optimal}. Thus, some assumptions are required. In this work we make the following minimal assumptions on the underlying distribution.

\begin{definition} \label{def:admis}
	We say that a distribution $\mathcal{D}$ supported on $\RR^d$ has a $q$-admissible law with the parameters $R_{\max},R_{\min}, r,L>0$, if the following four conditions hold:
	\begin{enumerate}
   \item For every $\theta \in \Sph$, $\Q(\langle X,\theta \rangle) \in [-R_{\max},R_{\max}]$. 
    \item For some $c \in \RR^d$, it holds that, for every $\theta \in \Sph$, $|\Q(\langle X, \theta\rangle) - \langle c, \theta \rangle| \geq R_{\min}$.
    \item For every $\theta \in \Sph$, the density $f_\theta$ of $\langle X, \theta \rangle$ exists and satisfies 
    $$f_\theta(t) \geq L \text{ for all }t \in [\Q(\langle X, \theta \rangle)-r, \Q(\langle X, \theta \rangle) +r].$$
    \item If $\{X_i\}_{i=1}^n$ are \emph{i.i.d.} as $\mathcal{D}$, then with high probability $\|X_i\|_2 \leq \mathrm{poly}(n,d)$, for all $i \in [n]$.

\end{enumerate}
	The set of all such distributions is denoted by $A_q(R_{\max},R_{\min}, r,L)$.	
\end{definition}

We now expand further on the claimed ``minimality'' of the above assumptions. First, note that assumptions (1) and (3) are actually necessary for consistent private quantile estimation. This is an outcome of the tight lower bound for the one-dimensional case (see the main result of \cite{tzamos2020optimal}), where it is established that the minimal sample complexity to privately learn one quantile of a distribution explodes to infinity whenever there is no bound on the possible quantile values ($R=R_{\max}=\infty$) or the distribution has no mass around the quantile ($L=0$). We highlight that the latter property (of having a positive density around the quantile) is a requirement for optimal estimation rates (and a standard assumption) even for \emph{non-private} one-dimensional quantile estimation (e.g., see \cite[Chs. 21, 25.3]{van2000asymptotic}). Now, assumption (2) is equivalent to the assumption that the floating body contains a ball of radius $R_{\min}$. Again, if $R_{\min}=0$, the floating body has an empty interior, and the floating body's estimation becomes a degenerate task. Finally, assumption (4) is a very mild boundedness technical condition that our data is arbitrarily polynomially-bounded, which almost does not hurt generality. For example, for ``heavy-tailed'' Cauchy random vectors this assumption holds for the polynomial $dn$.
\\ 
\paragraph{Private many-quantiles estimation.} 
Let us introduce a useful convention: we identify a fixed vector $X=(X_1,\ldots,X_n) \in (\RR^{d})^n$ with a random vector in $\RR^d$, chosen uniformly from $(X_1,\ldots,X_n)$. In other words we identify a ``sample'' $X$ with the empirical distribution over its $d$-dimensional coordinates. Now,
in \cite{anderson2020efficiency} it was established that for any symmetric log-concave measure $\mathcal{D}$, if one has $n$ i.i.d. samples  $X=(X_1,\ldots,X_n) \in (\RR^{d})^n$ from $\mathcal{D}$, then the empirical measure on the samples satisfies for $Y \sim \mathcal{D}$ with probability at least $0.9,$
\begin{align}\label{eq:delta_q, alpha}
    \qd(X,Y):=\sup_{\theta \in \Sph} |\Q(\langle X, \theta \rangle)-\Q(\langle Y, \theta \rangle)| \leq \alpha,
\end{align}
for some $n=\tilde{O}(d/\alpha^2)$.\footnote{Here, and throughout the paper, $\tilde{O}$ means we omit logarithmic factors, in all parameters.} One can easily generalize the above result for any $q$-admissible law, as described in Definition \ref{def:admis}(see Section \ref{sec:admiss}).

We now proceed with an observation. For any integer $M>0$, and for any subset of $M$ directions, say that one wants to estimate the $q$-quantiles of the distribution along each of the $M$ directions. A simple ``low-dimensional'' approach would be as follows; one can take $n$ samples and then take the empirical $q$-quantiles of the projection of the samples along each of the $M$ directions. Using standard concentration results, having $\tilde{O}(1/\alpha^2)$ fresh samples per direction suffices for estimation in the corresponding direction, leading to the desired estimation guarantee with access to $n=\tilde{O}(M/\alpha^2)$ samples. We note that for each specific direction $\tilde{O}(1/\alpha^2)$ samples are known to be necessary (see e.g. \cite[Proposition 3.7]{tzamos2020optimal}). Comparing with the guarantee of \cite{anderson2020efficiency} in \eqref{eq:delta_q, alpha} this leads to an interesting high-dimensional statistical phenomenon: without privacy constraints, \emph{simultaneously} estimating $M>d$ marginal quantiles in high-dimensions requires less samples than the natural (but perhaps naive) approach of projecting the points and then calculating the empirical quantile on each direction.

Our first contribution is to reveal the private analogue of the above high-dimensional phenomenon. Consider the task of \emph{private estimation of $M$ quantiles}, i.e. of $\Q(\langle Y, \theta_i \rangle), Y \sim \mathcal{D}$ for $M$ different directions $(\theta_i)_{i=1}^M$. The ``naive'' approach would be, as before, to first project along the different directions, but then, instead of calculating the (non-private) empirical quantiles of the projected points, one would apply the optimal private algorithm of \cite{tzamos2020optimal} to each projection, separately.\footnote{Note that the algorithm from \cite{tzamos2020optimal} is technically only designed for median estimation, but with the natural modifications the same algorithm and analysis can be straightforwardly extended for $q$-quantile estimation.} This would lead to the guarantee in \eqref{eq:delta_q, alpha} with sample complexity (assuming, for simplicity, \emph{here and throughout the contribution section} all parameters besides $M,\eps,\alpha$ are constant):
\begin{align}\label{rate_1d} 
n=O\left(  \frac{M}{ \alpha^2}+\frac{ M}{\eps\alpha }\right).
\end{align}
Our first main result is an improved private rate for the estimation of $M$ quantiles which appropriately privatizes the high-dimensional result of \cite{anderson2020efficiency}.
\begin{theorem}[informal]\label{thm1_contrib}
There exists an $\eps$-differentially~private algorithm that draws
\[
n=\tilde{O}\left( \frac{d}{ \alpha^2} +\frac{M}{\eps \alpha } \right)
\]
samples from any admissible distribution $Y \sim \mathcal{D}$ and produces an estimate $\hat m\in \RR^M$ that satisfies $$\max\limits_{i=1,\ldots,M} |\hat m_i - \Q(\langle Y, \theta_i \rangle)| \le \alpha$$ with probability at least $0.9$.
\end{theorem} 
Notice that the rate described in Theorem \ref{thm1_contrib} is better than the rate in \eqref{rate_1d} when $M>d$ and $\eps>\alpha$.  This constitutes indeed a private analogue of the high-dimensional phenomenon mentioned above; privately estimating $M>d$ quantiles can require much less samples than simply projecting and applying the one-dimensional machinery.
\\
\paragraph{Private interior point.}
Next, we attempt to privately produce point estimates of the floating body $F_q(Y)$ itself. Our first task is to output an interior point of $F_q(Y)$, a relatively easy task without privacy concerns \cite{anderson2020efficiency}, with a clear statistical motivation of outputting a ``typical datapoint'' in high-dimensions. We show that with polynomial in $d$ samples one can achieve such a guarantee with small error. 
There has been closely related (but not fully comparable) work on this task in the worst case regime, which we discuss in depth in Section \ref{sec:related}.

For our result we leverage the Steiner point of the floating body, denoted $S(F_q(Y))$ (see \eqref{eq:steinerdef} for the definition). The Steiner point has been widely studied in the context of \emph{Lipschitz selection} \cite{Benyamini2000geometric}. In the problem of Lipschitz selection, one looks to consistently select a point from the interior of each convex body, in a way that is robust to perturbations of the body. The Steiner point is celebrated for being the optimal Lipschitz selector in this setting. Just recently, this construction proved to be instrumental in the resolution of the \emph{chasing nested convex sets} problem \cite{bubeck2020chasing,sellke2020chasing,argue2021chasing}, a problem in online learning related to Lipschitz selection. It is precisely this optimal robustness that makes the Steiner point insensitive to individual sample points, a desirable property in the design of DP algorithms. To the best of our knowledge, this is the first time the Steiner point and its properties have been used it to build differentially private estimators, and we believe it might be useful in other similar problems.

\begin{theorem}[informal]\label{thm2_contrib}
There exists an $\eps$-differentially~private algorithm that draws 
\[
n=\tilde{O}\left( \frac{d}{ \alpha^2} +\frac{d^{2.5}}{\eps \alpha } \right)
\]
samples from any admissible distribution $Y \sim \mathcal{D}$ and produces an estimate $\hat m \in \RR^d$ that satisfies 
$$\|\hat m - S(F_q(Y))\|_2 \le \alpha$$ 
with probability at least $0.9$. 
\end{theorem}
In comparison, the best known dependence of $n$ on the dimension
in the worst case literature is of order $\tilde{O}(d^4)$ for pure $\eps$-DP \cite{Sharir2020ConvexHull} and of order $\tilde{O}(d^{2.5})$ for the weaker notion of approximate $(\eps, \dpdelta)$-DP \cite{Stemmer19Centerpoint}.\footnote{See Section \ref{sec:prelims} for a definition and comparison between the different privacy notions.} Thus, Theorem \ref{thm2_contrib} achieves essentially the same dependence as the best known approximate DP result while enjoying the stronger guarantees of pure DP. We note again that the results are not fully comparable; see Section \ref{sec:related} for a thorough discussion.
\\
\paragraph{Private sampling from the floating body.}
Outputting an, a-priori, arbitrary point from the floating body comes with possible drawbacks, for example biasing the output point towards or away from its boundary. This can lead to inaccurate statistical conclusions if a more detailed description of the ``typical datapoints'' is needed. For this reason, we undertake the harder task of outputting a \emph{uniform sample}.

The \emph{sampling from convex bodies} problem has a long and rich history with many deep results, see \cite{kannan1995isoperimetric,kannan1997random,lovasz2006simulated} for some prominent examples. In many cases, the computation preformed by the sampling algorithm reduces to iteratively projecting points into the convex body, along the trajectory of an appropriate Markov chain \cite{bubeck2018sampling,lehec2021langevin}.

It is precisely the idea described above which serves as the working engine behind our private sampling algorithm. To cope with privacy constraints, we capitalize on known robustness properties of the projection operator (see \eqref{eq:projdef}) on convex bodies, as proved in \cite{attouch1993quantitative}. Remarkably, when the projected point is fixed and one considers the convex body as the variable, the projection operator is known to be 1/2-H\"older with respect to the Hausdorff distance. This allows us to build a noisy projection oracle for the floating body, in a private fashion. Combining the noisy oracle with known results leads to a new sampling algorithm, which operates under pure differential privacy, and with guarantees in the quadratic Wasserstein distance, $W_2$ (see Section \ref{sec:wass} for the definition).
\begin{theorem}[informal]\label{thm3_contrib}
Let $\mu_{F_q(Y)}$ be the uniform measure on the floating body. There exists an $\eps$-DP algorithm that draws
\[
n=\tilde{O}\left( \frac{d^2}{ \alpha^{14}} +\frac{d^{4}}{\eps^2 \alpha^{8} } \right)
\]
samples from any admissible distribution $Y \sim \mathcal{D}$, and whose output, $\hat \mu_{F_q(Y)}$, satisfies 
$$\frac{1}{d}W^2_2(\hat \mu_{F_q(Y)}, \mu_{F_q(Y)}) \leq \alpha.$$
\end{theorem}

We note that the $\frac{1}{d}$ normalizing factor in Theorem \ref{thm3_contrib} serves as the correct scale for the $W_2$ metric which typically scales as square-root of the dimension, see \cite{lehec2021langevin} for further discussion on this.

\subsection{Further Comparison with \texorpdfstring{\cite{Stemmer19Centerpoint, Sharir2020ConvexHull}}{[]}}
\label{sec:related}
The most closely related work to our point estimators of the floating body (Theorems \ref{thm2_contrib} and \ref{thm3_contrib})  are by Beimel et al.~\cite{Stemmer19Centerpoint} and Kaplan et al.~\cite{Sharir2020ConvexHull}. Both of these works address under worst-case input the private \emph{interior point} problem: namely, given a (worst case) set of $n$ points $x_1, \ldots, x_n$ in $d$ dimensions, the task is to privately output a point that lies within the convex hull of these points. While the non-private analogue is trivial (since one can always output, say, $x_1$), the problem becomes interestingly non-trivial when privacy considerations are introduced. To achieve such guarantees, the set of points is assumed to be a subset of a finite grid of possible points $U^d$ (where $U$ is some finite universe). This is a necessary condition for the existence of successful private estimators under worst-case assumptions  \cite{BeimelPurevsApprox2016,BunThreshold15}), whereas this assumption is \emph{not required} for our results. On the other hand, due to the use of the extension lemma, our algorithms do not have any explicit bound on the running time (see Section \ref{sec:approach} below), unlike the situation in the worst case works, whose running time is generally exponential in $d$. 

The first work \cite{Stemmer19Centerpoint} solves the private interior point problem under approximate $(\eps, \dpdelta)$-differential privacy, a weaker privacy guarantee than $\eps$-DP, assuming that the dataset is of size at least some $n = \tilde{O}(d^{2.5} \log^*(|U|))$ (in the $\tilde{O}$ term here we suppress dependencies in $\eps, \dpdelta$, and lower order terms). The second work~\cite{Sharir2020ConvexHull} solves the problem under pure differential privacy (as in our paper) for $n = \tilde{O}(d^4\log |U|)$, and also obtains an improved $O(n^d)$ running time using $O(d^4\log |U|)$ samples under approximate DP. The $n = \tilde{O}(d^4\log |U|)$ bound is currently the best known sample complexity under pure DP for this worst-case task. In both \cite{Stemmer19Centerpoint} and \cite{Sharir2020ConvexHull}, the authors' approach is to output a point of high Tukey-depth with respect to the set of the $n$ points, which can be interpreted as outputting an element of the $q$-floating body of the $n$ points for some appropriate, and perhaps data-dependent, value of $q$.

While our considered settings are similar but incomparable, we recall that our Theorem \ref{thm2_contrib} obtains pure $\eps$-DP guarantees that for any $q$ output a point of the $q$-floating body with $\tilde{O}(d^{2.5})$ samples. 
Finally, we note that matching the sample complexity of the best $(\eps,\dpdelta)$-approximate private algorithm with an $\eps$-private algorithm has been challenging in the privacy literature, e.g., in the context of high-dimensional Gaussian mean estimation \cite{Kamath0SU19, HeavyTailed}, and in fact provably impossible in the worst-case context for many problems of interest; see \cite{Stemmer19Centerpoint} for more details.

\subsection{General Approach} \label{sec:approach}

Our approach towards building all $\eps$-private algorithms in this work is based on an appropriate use of a Lipschitz extension tool called the Extension Lemma \cite{borgs2018private, borgs2018revealing}. This lemma was also the main tool behind the construction of the private median estimator in \cite{tzamos2020optimal}. One of the main hurdles in constructing sample-efficient $\eps$-differentially private algorithms under distributional assumptions is that the privacy constraint needs to hold for all input data-sets, while the accuracy is based on the behavior of the algorithm on ``with-high-probability'' or typical data-sets. The Extension Lemma completely resolves this issue and allows the algorithm designer to focus on ensuring privacy solely on the typical data-sets.\footnote{While this work is entirely focused on sample complexity guarantees, it is worth mentioning that the Extension Lemma comes in principle with no explicit termination time guarantee (see \cite{borgs2018private} and Appendix \ref{prem:extension}).} Then a (privacy-preserving) Lipschitz Extension argument takes care of extending the private algorithm on typical-inputs to a private algorithm on all possible inputs, while maintaining the original algorithm's output on the typical inputs (see Proposition \ref{extension_prop} for more details). Leveraging this tool, our results are based on a two-step procedure.

\begin{itemize}
    \item[1.] First we find the (non-private) estimator of interest for each task. For example, this is the Steiner point of the floating body for Theorem \ref{thm2_contrib} and the projection of a point to the floating body for Theorem \ref{thm3_contrib}. Then we define a \emph{typical set} of possible input datasets $X=(X_1,\ldots,X_n) \in (\RR^d)^n$ which is realized with high probability over all inputs drawn from admissible distributions. Following this we establish Lipschitzness properties of the non-private estimator seen as a function of the input $X$ \emph{with respect to the Hamming distance between two inputs} of the (non-private) estimator, when defined only on inputs from the typical set. Such robustness properties can be non-trivial and to do so we use interesting tools from convex geometry to establish them for the Steiner point \cite{Benyamini2000geometric}, and for the projection of a point \cite{attouch1993quantitative}. Notably, none of these estimators is (non-trivially) Lipschitz with respect to the Hamming distance of the input, unless the input is on the typical set. Using now the Lipschitzness properties combined with classic differential privacy ideas we can ``privatize'' the estimator by applying the (flattened) Laplace mechanism \cite{tzamos2020optimal}, while ensuring high accuracy.
    
    \item[2.] With the Extension Lemma, we extend our estimator to be defined and private on all possible inputs. The Extension Lemma also guarantees that the estimator remains identical when the input is from the typical set. These allow to obtain a globally private algorithm with the same accuracy guarantee.
\end{itemize}
\subsection{Future Directions}
Our work suggests several future directions. First, our estimators come with no polynomial termination time guarantee since they are based on applying the generic Extension Lemma. Yet, in the one dimensional case, \cite{tzamos2020optimal} showed that the Extension Lemma can be computed in polynomial time when applied to the empirical median on an appropriate typical set. Given the analogies between our works, it is an important open problem to see if the high-dimensional estimators we propose can also be implemented in polynomial-time. Second, our work leverages classical robust convex geometric tools to construct private algorithms, which allow to demonstrate non-trivial high-dimensional phenomena. It will be interesting to see if these newly introduced tools in the privacy literature can be further used to reveal other non-trivial high-dimensional results. 

\section{Why Floating Bodies?}\label{sec:float_prelim}
The main object of interest in the present work is the \emph{convex floating body}. This is a well known construct in convex geometry, which takes its name from the work of Archimedes (see e.g. \cite{archimedes_2009}), and shares fascinating connections with non-parametric statistics. Here we further elaborate on the motivation to study floating bodies as a natural, robust high-dimensional statistical object. We refer the interested reader to the survey of Nagy et al.~\cite{nagy2019halfspace} for an extensive discussion of its statistical relevance.

Let us first recall the definition of the convex floating body. 
Let $q \in [0,1]$, and let $Y$ be a random variable in $\RR$. 
The \emph{$q$-quantile} of $Y$, which we denote $\Q(Y)$, is
$$\Q(Y) := \arg\inf\{t\in \RR |\PP\left(Y \leq t \right) \geq q\}.$$
That is, $\Q(Y)$ is the infimum value satisfying $\PP\left(Y \leq \Q(Y)\right) \geq q.$ Now, for a random vector $X$, in $\RR^d$, the (convex) $q$-floating body, $F_q(X)$, is defined by \[
F_q(X)= \bigcap\limits_{\theta\in \Sph} \left\{x \in \RR^d : \langle x, \theta\rangle \leq \Q(\langle X, \theta\rangle)\right\},
\]
where we denote by $\Sph \subset \RR^d$ the unit Euclidean sphere, in $d$-dimensions. In other words, the convex floating body is the intersection of all halfspaces containing at least a $1-q$ fraction of the mass of the distribution. 
See Figure~\ref{fig:fb} for a pictorial representation of the floating body of a data-set and \cite[Figures 1-4]{nagy2019halfspace} for illustrations of floating bodies in various interesting contexts.
\begin{figure}
\centering
\includegraphics[width=0.6\textwidth]{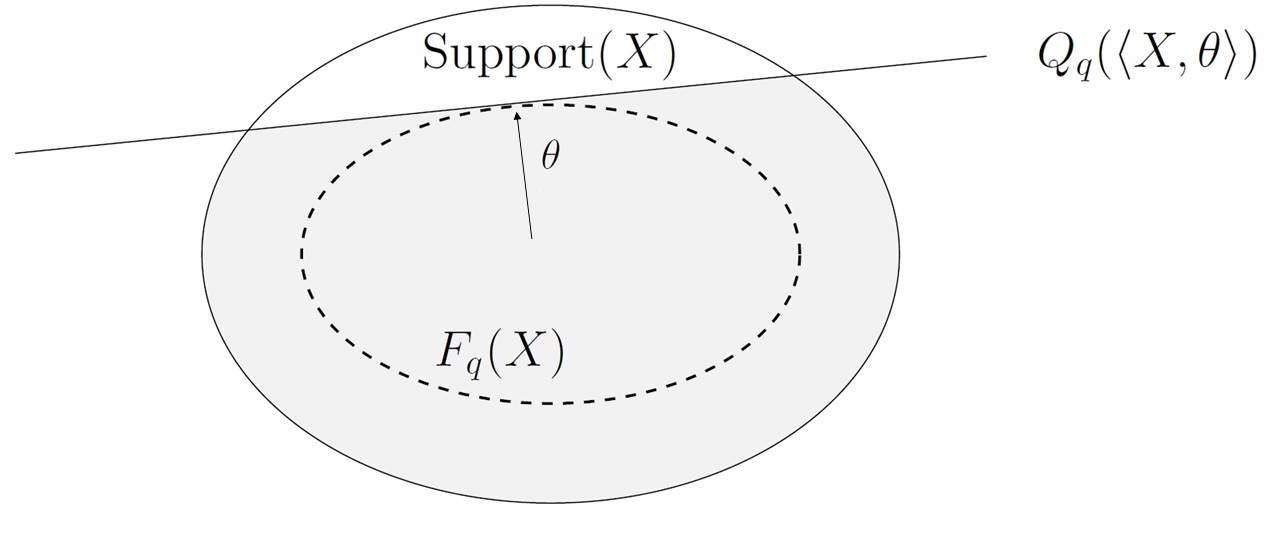}
\caption{The floating body of a dataset $X$.}
\label{fig:fb}
\end{figure}

The floating body has the following desirable properties. 
\\ 
\paragraph{Existence.} The convex floating body always exists for $q \geq 1-1/(d+1)$. %
This follows from the fact that for any set $S$ of $n$ vertices in $d$ dimensions, there is some point with Tukey depth 
$n/(d+1)$ with respect to $S$, which in itself is a standard implication of Helly's theorem from convex analysis (e.g., \cite{Stemmer19Centerpoint}).

The above result holds under worst case assumptions, and can be overly pessimistic for many types of realistic data distributions (in the sense that the maximum Tukey depth guaranteed, $n/(d+1)$, depends inversely on $d$).
However,
for many distributions of practical interest exhibiting some sort of ``niceness'' properties (or ``admissibility" as we say in the current work -- see Definition \ref{def:admis}, that ensures $F_q(\mathcal{D})$ contains a ball of radius $R_{\min}$ centered at $c$), much stronger guarantees are known. One example is the wide and important family of centrally symmetric log-concave distributions, which includes for example the Gaussian, uniform, and Laplace distributions, among many others. When the random vector $X$ is generated from any such admissible distribution, it suffices to take $n$ polynomial in $d$ to ensure that the empirical $q$-floating body is non-empty for any fixed $q > 1/2$ (independently of the dimension $d$); see 
Section \ref{sec:appendix_admissible} for more details. 
\\ 
\paragraph{The floating body is a natural high dimensional statistical construction.} 
Privately outputting descriptive statistics of a given dataset is among the most fundamental tasks in the privacy literature.
Indeed, a large body of very recent work in the differential privacy literature is devoted to privately estimating quantiles in one dimension (e.g., \cite{brunel2020propose,tzamos2020optimal,gillenwater2021icml, SchnappQuantiles2022, Alabi2022, Lalanne2022}), or  uses the interquantile range to privately output meaningful measurements of the standard deviation of one dimensional distributions \cite{DworkLei09}.

However, one-dimensional statistics (applied to projections of some high-dimensional data) cannot generally capture the complexity of high dimensions. Suppose, as \emph{a running and motivating example}, that we maintain a large high-dimensional database, where each (high-dimensional) entry represents the feature vector of a single user in the database. Naturally, one might want the ability to privately generate artificial users that exhibit the ``typical'' behavior of actual users in the database, without compromising on users' privacy. An important application is private generation of high quality synthetic data for training machine learning models on the database, which should be accurate enough to work well on actual users, yet maintain the privacy on existing users in the database (e.g., \cite{Gaboardi14,yoon2018pategan,comparativeSynthetic21,McKennaSynthetic19}).

Privately sampling from the floating body of the random variable $X \sim \mathcal{D}$ (for an unknown distribution $\mathcal{D}$) is arguably the most statistically principled approach to privately generating a large and diverse yet "representative" collection of points (which is the type of access needed for the synthetic data generation application) from the unknown $\mathcal{D}$, given only sample access to it.
\\ 
\paragraph{Robustness.}
The existence of outliers in the data is one of the most challenging aspects to statistics and machine learning in high dimensions. One of the main difficulties is that there is no canonical definition of what constitutes an outlier in the data. The convex floating body offers a very simple, nonparametric interpretation of ``central points'' of the distribution: these are precisely all points in the floating body $F_q$ (where $q$ can possibly depend on the data), i.e., all points that in every direction fall in the $q$-interquantile range. This point of view affords a high-dimensional interpretation for outliers; A point $x\in \RR^d$ is an outlier in direction $\theta \in \Sph$, if $\langle x, \theta \rangle > h_{F_q(\mathcal{D})}(\theta)$, where $h_{F_q(\mathcal{D})}$ is the support function of $F_q(\mathcal{D})$, as described below in \eqref{eq:support}. One appeal of using the support function, as opposed to the quantile in direction $\theta$, is that the latter only depends on a marginal, while the former depends on the joint high-dimensional distribution, and so represents a more integrative decision rule.
\\ 
\paragraph{Rich convex geometry foundations.}
As is demonstrated throughout the paper (and in more detail, e.g., in the survey \cite{nagy2019halfspace}), there is a very rich understanding of convex floating bodies in high dimensions from the convex geometry perspective. Indeed, objects and questions of this type have been systematically studied in the last two hundred years; the earliest modern work is Dupin's book from 1822 \cite{dupin1822applications}. As described in the survey, the interplay between different notions of symmetry and depth arising in convex bodies and log concave measures (which are, in many ways, the natural measure-theoretic generalization of a convex body) gives rise to deep and interesting mathematics. Thus, working with the convex body provides us access to this rich literature without the need to establish a mathematical theory from scratch.

\section{Preliminaries}\label{sec:prelims}

\subsection{Extension Lemma}
Let us consider an arbitrary $\eps$-differentially private algorithm defined on $n$ samples in $\mathbb{R}^d$ as input and belonging in some set $\mathcal{H} \subseteq (\mathbb{R}^n)^d$. Then the Extension Lemma guarantees that the algorithm \emph{can be always extended} to a $2\eps$-differentially private algorithm defined for arbitrary input data in $(\mathbb{R}^n)^d$ with the property that if the input data belongs in $\mathcal{H}$, the distribution of output values is exactly the same with the original algorithm. We note that the result in \cite{borgs2018private} is ``generic'', in the sense of applying to any input metric space and output probability space, but here we present it for simplicity only on inputs from $\mathbb{R}^n$ and equipped with the Hamming distance, $d_H.$ Formally the result is as follows.

\begin{proposition}[The Extension Lemma, Proposition 2.1, \cite{borgs2018private}] \label{extension_prop} 
Let $\hat{\mathcal{A}}$ be an $\eps$-differentially private algorithm designed for input from $\mathcal{H} \subseteq (\mathbb{R}^d)^n$ with arbitrary output measure space $(\Omega,\mathcal{F})$. Then there exists a randomized algorithm $\mathcal{A}$ defined on the whole input space $(\mathbb{R}^d)^n$ with the same output space which is $2\eps$-differentially private and satisfies that for every $X \in \mathcal{H}$, $\mathcal{A}(X) \overset{d}{=}  \hat{\mathcal{A}}(X)$.
\end{proposition}

\subsection{Useful Distances}
The following pseudo-metric simply measures the maximum difference between the  $q$-quantiles of two (random) vectors.
\begin{definition}[$q$-distance] \label{def:qdist}
	Let $q \in (\frac{1}{2},1)$ and let $A \subset \Sph$. If $X$ and $Y$ are two random vectors in $\RR^d$, their $q$-distance with respect to $A$ is defined by
	$$\qd(X,Y;A) := \sup\limits_{\theta \in A}  \left|\Q(\langle X, \theta\rangle) - \Q(\langle Y, \theta\rangle)\right|.$$
	For $A = \Sph$, we write $\qd(X,Y)$ as an abbreviation for $\qd(X,Y;A)$.
\end{definition}In the sequel, we consider two types of the set $A$, in Definition \ref{def:qdist}: (i) finite sets of directions, which correspond to privately estimating a finite (but potentially large) number of quantiles; and (ii) the whole set of possible directions $A = \Sph$. In this case where $\qd(X,Y;A) = \qd(X,Y)$, bounds on $\qd(X,Y)$ are strong enough to allow estimation of high-dimensional global statistics of $F_q(X)$, such as the Steiner point of the body or the projection operator to the body.

To further elaborate on this remark, let us recall that there is another natural convex geometric way to measure distances between convex bodies, \emph{the Hausdorff distance}. For two convex bodies $K,K'\subset \RR^d$, their Hausdorff distance is 
$\qha(K_1, K_2) :=\inf\{t > 0 | K \subset K' + t B_d \text{ and } K' \subset K+ t B_d \},$
where $B_d$ is the unit Euclidean ball in $\RR^d$. The following connection between $\qha$ and $\qd$ holds.
\begin{lemma} \textnormal{\cite[Lemma 5]{brunel2018methods}} \label{lem:brunel}
	Let $q \in (\frac{1}{2}, 1)$ and let $X$ and $Y$ be two random vectors in $\RR^d$. Suppose that $F_q(X)$ contains a ball of radius $R_{\min}$ and is contained in a ball of radius $R_{\max}$, for some $R_{\max}, R_{\min} >0$. If $\qd(X, Y) \leq \frac{R_{\min}}{2}$ then,	$\qha(F_q(X), F_q(Y)) \leq \frac{3R_{\max}}{R_{\min}}\qd(X,Y).$
\end{lemma}

 Let us also introduce a useful object called the \emph{support function} of a convex body. If $K \subset \RR^d$ is a convex body its support function, $h_K: \Sph \to \RR$, is defined by
\begin{equation} \label{eq:support}
	h_K(\theta) = \max\limits_{x\in K} \langle x, \theta\rangle.
\end{equation}
Using the support function we can give an (alternative) functional definition of the Hausdorff distance, with the following equivalence (see \cite{asymptotic2015avidan}):
\begin{equation} \label{eq:haus}
	\qha(K, K') = \sup\limits_{\theta \in \Sph}|h_K(\theta) - h_{K'}(\theta)|.
\end{equation}
when $K$ and $K'$ are convex bodies.

\subsection{Approximate Differential Privacy}

Throughout the paper we refer to a notion, similar to ``pure'' $\eps$-differential privacy (Definition \ref{definition:epsilon-privacy}), called ``approximate'' $(\eps,\dpdelta)$-differential privacy. We now formally define it.
\begin{definition}\label{definition:epsilon-delta-privacy}
A randomized algorithm $\mathcal{A}$ is $(\eps,\dpdelta)$-differential private if for all subsets $S \in \mathcal{F}$ of the output measurable space $(\Omega, \mathcal{F})$ and $n$-tuples of samples $X_1,X_2 \in (\mathbb{R}^d)^n$ it holds \begin{equation} \label{delta-privdfn}\mathbb{P}\left(\mathcal{A}(X_1) \in S\right) \leq e^{\eps d_H(X_1,X_2)}\mathbb{P}\left(\mathcal{A}(X_2) \in S\right)+\dpdelta.\end{equation}
\end{definition}Clearly $(\eps,\dpdelta)$-differential privacy is a \emph{weaker} notion to $\eps$-differential privacy, in the sense that for any $\eps,\dpdelta>0$ any $\eps$-differentially private algorithm is also an $(\eps,\dpdelta)$-differentially private algorithm.

\section{A Meta Algorithm for H\"{o}lder Queries}

We are now ready to present our main result. In the introduction, we highlighted three results of our work, all of which follow from a meta-theorem resulting in a differentially private algorithm for querying generic H\"older statistics of the floating body with respect to the $\qd$ norm. Since our result is general, it requires certain notation.
\\ 
\paragraph{Approximate H\"older queries}
Let $p \in [1,\infty],$ $h \in (0,1]$, and $M \in \NN$. Denote by $\mathcal{C}_d$ the space of all convex bodies in $\RR^d$, and assign $M$ to be the dimension of the output space, equipped with the $L_p$-norm. We  say that a map $f: \mathcal{C}_d \rightarrow \mathbb{R}^M$ is $h$-H\"older (with constant $K>0$, and with respect to a set $A \subset \Sph$), or simply Lipschitz when $h = 1$, if
\begin{equation} \label{eq:holder}
	\|f(F_q(X)) - f(F_q(Y))\|_p \leq  K\qd(X,Y; A)^h,
\end{equation}
where $X$ and $Y$ are random vectors. One example to keep in mind is when $X$ and $Y$ are both the empirical distributions of two samples that satisfy $d_H(X,Y) = 1$. In this case we establish that, when drawn from admissible distributions, with high-probability $\qd(X,Y; A)$ is small (see Lemma \ref{lemma:sensitivity-q}). Thus, \eqref{eq:holder} will imply a low sensitivity condition for $f$, a desirable property for the design of differentially private algorithms to approximate $f$.

It turns out that many of the queries we study, like the Steiner point, are H\"older with respect to the Hausdorff distance, not the $\qd$ metric. In light of Lemma \ref{lem:brunel}, it will sometimes be convenient to restrict the domain in which \eqref{eq:holder} holds. In particular, for a fixed admissible class of measures, $A_q(R_{\max}, R_{\min},r, L)$, we shall enforce the condition that the floating bodies contain, and are contained in, a ball, as well as require some a-priori upper bound on the $\qd$ distance.
 \begin{definition}[Approximate H\"older functions for the class $A_q(R_{\max}, R_{\min},r, L)$] \label{def:approximateHolder}
We say that $f$ is \emph{approximate $h$-H\"older} if for all $X,Y$, which satisfy that for some $a \in \RR^d$, which could depend on $X$,
 $$B(a,R_{\min}/2) \subseteq F_q(X) \subseteq B(0,R_{\max}+r) \text{ and } \qd(F_q(X),F_q(Y)) \leq \frac{R_{\min}}{4},$$ \eqref{eq:holder} holds. 
 \end{definition}

 Since our private algorithm will be extended from a restricted algorithm on a typical set (recall the plan from Section \ref{sec:approach}) using Proposition \ref{extension_prop}, there will be no loss of privacy when considering approximate H\"older functions, as long as the desired conditions hold with high probability over the sample.
\\ 
\paragraph{The main result.}
We are now prepared to state our main theorem. All results mentioned in the preceding sections will follow by working with suitable approximate H\"older functions.
\begin{theorem} \label{thm:main}
	  Fix $q \in (1/2,1)$ and assume that $\mathcal{D} \in A_q(R_{\max}, R_{\min},r,L)$. Further, for $h,K >0$ and $A \subset \Sph$, let $f: \mathcal{C}_d \rightarrow \mathbb{R}^M$ be an approximate $h$-H\"older function with constant $K$, with respect to $A$. Then, there is an $\eps$-differentially private algorithm $\mathcal{A}$ which for input $X=(X_1,\ldots,X_n)$, sampled \emph{i.i.d} from $\mathcal{D}$, satisfies for all $\alpha<\min\{1,K\}\frac{\min\{r,R_{\min}\}}{2}$ that
	$\PP(\|\mathcal{A}(X)-f(F_q(\mathcal{D}))\|_p \leq \alpha) \geq 1-\beta,$
	for some $n$ such that
	\begin{align*}
	    n = \tilde{O}\left(\frac{K^{2/h}\left(d + \log\left(\frac{4}{\beta}\right)\right)}{\alpha^{2/h}L^2}
	    +\frac{\C K^{1/h}(\log \left(\frac{1}{\beta}\right)^{1/h}M^{1/h})}{(\eps \alpha)^{1/h} L }
	    +\frac{\C M^{1/h}}{(\eps \min\{r,R_{\min}\})^{1/h}L}\right),
	\end{align*}
	where $\C =\min\{d,\log |A|\}+ \log\left(\frac{1}{\beta}\right).$
\end{theorem}
The obtained rate may seem complicated; this is to be expected, given the generality of our results and the number of parameters. In the next section we demonstrate several concrete uses of Theorem \ref{thm:main}, which show how the rate simplifies in various interesting settings.
\section{Applications}
Here we describe the main applications of Theorem \ref{thm:main}. As mentioned we apply the theorem using suitable approximate H\"{o}lder functions.
\vspace{-1em}
\\ 
\paragraph{Simultaneous estimation of quantiles}
Fix $q \in (0,1)$ and let $A \subset \Sph$, with $|A|=M$. Define the multiple $M$-query function which on a random vector $X$, equals
$f_A(X) = \{\Q(\langle X, \theta\rangle)\}_{\theta \in A}.$

It is immediately seen that $f_A$ is a Lipschitz function\footnote{Strictly speaking, $f_A$ is not a function of the floating body, but of the sample itself. However, with a trivial adaption, it still fits nicely within our framework.}, with constant $1$, in the $L_\infty$ ($p=\infty)$ norm:
\begin{equation} \label{eq:multiquerieslip}
	\|f_A(X)-f_A(Y)\|_\infty = \max_{\theta \in A} |\Q(\langle X, \theta\rangle - \Q(\langle Y, \theta\rangle| = \qd(X,Y;A).
\end{equation}
We thus have the following result.
\begin{corollary}
	Let $\mathcal{D} \in A_q(R_{\max}, R_{\min},r,L)$ be an admissible measure on $\RR^d$. Then, there is an $\eps$-differentially private algorithm $\mathcal{A}$ which for input $X=(X_1,\ldots,X_n)$, sampled \emph{i.i.d} from $\mathcal{D}$, satisfies for all $\alpha<\frac{\min\{r,R_{\min}\}}{2}$,
	$\PP(\|\mathcal{A}(X)-f_A(\mathcal{D})\|_\infty \leq \alpha) \geq 0.9,$
	for $$n=\tilde{O}\left(\frac{d}{\alpha^2L^2}+\frac{M}{\eps L \alpha}+\frac{M}{\eps L \min\{r,R_{\min}\}}\right).$$
\end{corollary}
\begin{proof}
	The observation in \eqref{eq:multiquerieslip} shows that $f_A$ is a Lipschitz function with constant $1$. The result now follows by invoking Theorem \ref{thm:main} for the $L_\infty$ norm, with $K,h = 1$, $\beta = 0.1$, and $\C = \log(M)$.
\end{proof}
\paragraph{Privately returning an interior point: the Steiner point.}
Given Theorem \ref{thm:main} it will be enough to show that one can select a point from a convex body in a Lipschitz way. This naturally leads to the Steiner point, a widely studied object in Lipschitz selection.

If $K\subset \RR^d$ is a convex body, we define its Steiner point by
\begin{equation} \label{eq:steinerdef}
    S(K):= d\int\limits_{\Sph} \theta h_K(\theta)d\sigma,
\end{equation}
where $\sigma$ is the normalized Haar measure on $\Sph$, and the support function of $K,$ $h_K: \Sph \to \RR$, is defined by
$h_K(\theta) = \max\limits_{x\in K} \langle x, \theta\rangle.$

The following result follows from well-known results in convex geometry and from Lemma \ref{lem:brunel}.
\begin{lemma} \label{lem:steiner}
	The Steiner point $S: \mathcal{C}_d \to \RR^d$ is an approximate Lipschitz function, with constant $6\sqrt{d}\frac{R_{\max}+r}{R_{\min}}$, which satisfies $S(K) \in K$ for every convex body $K$.
\end{lemma}
\begin{proof}
We first show that, for every convex body $K$, $S(K) \in K$. Indeed, define $f_K(\theta) = \arg\max\limits_{x\in K}\langle x,\theta\rangle$. Observe that $\nabla h_K=f_K$. Hence, a straightforward application of the divergence theorem (see \cite[Chapter 6]{prezesl1989continuity}) gives:
\[S(K):= d\int\limits_{\Sph} \theta h_K(\theta)d\sigma = \int\limits_{\Sph} f_K(\theta)d\sigma,\]
So, $S(K)$ is a convex combination of $f_K(\theta)$. By definition, for every $\theta \in \Sph$, $f_K(\theta)\in K$ which implies, through convexity, $S(K) \in K$.
To prove that $S$ is Lipschitz, let $K' \subset \RR^d$ be any other convex body. We have
\begin{align*}
	\|S(K) - S(K')\|_2 &= \sup\limits_{u \in \Sph} \langle u, S(K) - S(K')\rangle \leq d\int\limits_{\Sph}|\langle u,\theta\rangle| |h_K(\theta) - h_{K'}(\theta)|d\sigma\\
	&\leq \qha(K,K')d\int\limits_{\Sph}|\langle u,\theta\rangle|d\sigma \leq \qha(K,K')d\sqrt{\int\limits_{\Sph}|\langle u,\theta\rangle|^2d\sigma}\\
	&=  \sqrt{d}\qha(K,K').
\end{align*}
The second inequality uses \eqref{eq:haus} and the third is Jensen's inequality. The last identity uses the well-known fact, obtained through symmetry, that the averaged square of a coordinate on $\Sph$ is $\frac{1}{d}$.
 Finally, let $X,X' \in (\RR^d)^n$ and assume that $F_q(X)$ contains a ball of radius $R_{\min}/2$, is contained in a ball of radius $R_{\max} +r$. centered at the origin, and that $\qd(X,X') \leq \frac{R_{\min}}{4}$.

 This allows us to invoke Lemma \ref{lem:brunel}, which, when coupled with the above bound, yields
 $$\|S(F_{q}(X)) - S(F_{q}(X'))\|_2 \leq \sqrt{d}\qha(F_{q}(X),F_{q}(X')) \leq 6\sqrt{d}\frac{R_{\max}+r}{R_{\min}}\qd(X,X')$$
 and concludes the proof.
\end{proof}
We immediately get the following corollary to Theorem \ref{thm:main}.
\begin{corollary} \label{cor:interior}
	Let $\mathcal{D} \in A_q(R_{\max}, R_{\min},r,L)$ be an admissible measure on $\RR^d$. Then, there is an $\eps$-differentially private algorithm $\mathcal{A}$ which for input $X=(X_1,\ldots,X_n)$, sampled \emph{i.i.d} from $\mathcal{D}$, satisfies for all $\alpha<\frac{\min\{r,R_{\min}\}}{2}$,
	$\PP(\|\mathcal{A}(X)-S(F_q(\mathcal{D}))\|_2 \leq \alpha) \geq 0.9,$
	for some $$n=\tilde{O}\left(d^2\frac{(R_{\max}+r)^2}{R^2_{\min}\alpha^2L^2}+d^{2.5}\frac{R_{\max}+r}{R_{\min}}\left(\frac{1}{\eps L \alpha}+\frac{1}{\eps L \min\{r,R_{\min}\}}\right)\right).$$
\end{corollary}
\begin{proof}
	Consider the Steiner point $S:\mathcal{C}_d \to \RR^d$. Lemma \ref{lem:steiner} states that $S$ is an approximate Lipschitz function with constant $6\sqrt{d}\frac{R_{\max}+r}{R_{\min}}$, for every $X,Y \in \mathcal{H}_C$. The result now follows by invoking Theorem \ref{thm:main} for the Euclidean norm, with $K = 6\sqrt{d}\frac{R_{\max}+r}{R_{\min}}$, $M = d$, $h =1$, $\beta = 0.1$, and $W = \tilde{O}(d)$.
\end{proof}
\paragraph{Private projection and sampling.} The sampling application is more involved than the previous applications and we defer the proofs to Section \ref{sec:proj and sampling}. Below we discuss the main ideas.

Let $K \subset \RR^d$ be a convex body. Define the projection operator, $P_K:\RR^d \to \RR^d$ by,
\begin{equation} \label{eq:projdef}
    P_K(x) = \arg\min\limits_y \{y \in K| \|y-x\|\}.
\end{equation}
We shall establish that, for admissible distributions $\mathcal{D}$ and for any point $x,$ one can privately estimate $P_{F_q(\mathcal{D})}(x)$ with polynomially many samples.
This follows by coupling a classic result in convex geometry, \cite[Proposition 5.3]{attouch1993quantitative}, concerning the stability of projections with Lemma \ref{lem:brunel}. 
\begin{lemma}\label{lem:lip_proj}
	Fix $x \in \RR^d$ and consider $P_K(x):\mathcal{C}_d\to \RR^d$. Then $P_K(x)$ is an approximate $\frac{1}{2}$-H\"older function with constant $5\sqrt{\frac{\left(\|x\|_2+R_{\max} + r\right)(R_{\max}+r)}{R_{\min}}}$.
\end{lemma}
Now, to sample from the body $K$, for $\eta > 0$, we define the following discretized Langevin process:
$$X_{t+1} = P_K(X_t + \eta g_t),\ \ \ \ X_0 = S(K),$$
where $\{g_t\}_{t \geq 0}$ are \emph{i.i.d.} standard Gaussians and $S(K)$ is the Steiner point of $K$.
It is well known (see for example \cite[Theorem 2]{lehec2021langevin}) that this process mixes rapidly, in the Wasserstein distance. By applying the known results about the mixing time of the Langevin process, and by taking account of the inherent noise introduced by the privacy constraints, we prove the following result.
\begin{corollary} \label{cor:sampling}
	Let $\mathcal{D} \in A_q(R_{\max}, R_{\min},r,L)$ be an admissible measure on $\RR^d$ and let $U_q$ be a random vector which is uniform on $F_q(\mathcal{D})$. Assume that $F_q(\mathcal{D})$ contains a ball of radius $R_{\min}$ around the Steiner point $S(K)$. Then, there is an $\eps$-differentially private algorithm $\mathcal{A}$ which for input $X=(X_1,\ldots,X_n)$, sampled \emph{i.i.d} from $\mathcal{D}$, satisfies for all $\alpha<\frac{\min\{1,r,R_{\min}\}}{2}$,
	$$\frac{1}{d}W^2_2(\mathcal{A}(X), U_q) \leq \alpha,$$ 
	for some
	$$n=\tilde{O}\left(\frac{\mathrm{poly}(R_{\max}+r+1)}{\mathrm{poly}(R_{\min})}\left(\frac{d^2}{\alpha^{14}L^2}+\frac{d^4}{\eps^2 \alpha^{8}L}+\frac{d^4}{\eps^2 \alpha^2\min\{r,R_{\min}\}^2L}\right)\right).$$
\end{corollary}
Corollary \ref{cor:sampling} requires that the floating body contains a ball, centered at the Steiner point. The reason for this assumption is that the discretized Langevin process is initiated at the Steiner point, and the distance from the initialization to the boundary of $F_q(\mathcal{D})$ will determine the mixing time of $X_t$. We chose the Steiner point as an the initial point because, by Theorem \ref{cor:interior}, we can privately approximate it. Moreover, the Steiner point tends to lie ``deeply'' in the interior of the convex body, although exact estimates seem to unknown for the general case \cite[Section 5.4]{schneider1993convex} (see \cite[Theorem 1.2]{shvartsman2004barycenter} for a similar construction with relevant guarantees).

To improve performance, one might impose some extra assumptions. For example, if the distribution $\mathcal{D}$ is symmetric around its mean, then the Steiner point will be at the center of $F_q(\mathcal{D})$. Another option is to assume that $F_q(\mathcal{D})$ contains a ball around a known point, like the origin, in which case we can initialize $X_0 =  0$. We chose to state Corollary \ref{cor:sampling} this way to make it as general as possible.

Finally let us note that Corollary \ref{cor:sampling} may be generalized to handle other log-concave distributions supported on $F_q(\mathcal{D})$. Such sampling schemes are related to optimization of convex functions and could be of further interest. We expand this discussion in Section \ref{sec:beyonduniformsampling}.

\section{Proof of Main Result: Theorem \ref{thm:main}}

In this Section we establish our meta-theorem, Theorem \ref{thm:main}, from which all our applications follow. 
For convenience, we first recall the statement of the theorem.
\begin{theorem}[Restated Theorem \ref{thm:main}]
	  Fix $q \in (1/2,1)$ and assume that $\mathcal{D} \in A_q(R_{\max}, R_{\min},r,L)$. Further, for $h,K >0$ and $A \subset \Sph$, let $f: \mathcal{C}_d \rightarrow \mathbb{R}^M$ be an approximate $h$-H\"older function with constant $K$, with respect to $A$. Then, there is an $\eps$-differentially private algorithm $\mathcal{A}$ which for input $X=(X_1,\ldots,X_n)$, sampled \emph{i.i.d} from $\mathcal{D}$, satisfies for all $\alpha<\min\{1,K\}\frac{\min\{r,R_{\min}\}}{2}$ that
	$\PP(\|\mathcal{A}(X)-f(F_q(\mathcal{D}))\|_p \leq \alpha) \geq 1-\beta,$
	for some $n$ such that
	\begin{align*}
	    n = \tilde{O}\left(\frac{K^{2/h}\left(d + \log\left(\frac{4}{\beta}\right)\right)}{\alpha^{2/h}L^2}
	    +\frac{\C K^{1/h}(\log \left(\frac{1}{\beta}\right)^{1/h}M^{1/h})}{(\eps \alpha)^{1/h} L }
	    +\frac{\C M^{1/h}}{(\eps \min\{r,R_{\min}\})^{1/h}L}\right),
	\end{align*}
	where $\C =\min\{d,\log |A|\}+ \log\left(\frac{1}{\beta}\right).$
\end{theorem}
We also recall the definition of an approximate H\"older function. We say that $f:\mathcal{C}_d \to \RR^M$ is an approximate $h$-H\"older function, with constant $K$, with respect to $A$, if,
\[
\|f(F_q(X)) - f(F_q(Y))\|_p\leq K\qd(X,Y;A)^h, 
\]
whenever $F_q(X)$ contains a ball of radius $\frac{R_{\min}}{2}$, is contained in a ball of radius $R_{\max} + r$, and $\qd(X,Y) \leq \frac{R_{\min}}{4}$.
\\ 
\paragraph{Organization} Before delving into the proof we provide a sketch in Section \ref{sec:sketch}, the proof is then split into several parts. The first part is to analyze a natural non-private estimator for the task (see Section \ref{sec:admiss}). In the second part, we begin ``privatizing" the non-private estimator. To do so, as described previously, we first restrict ourselves to a ``typical'' subset of possible inputs (described in Section \ref{sec:typical}). On the typical subset we apply to the non-private estimator a flattened Laplace mechanism, and calculate it's accuracy (see Section \ref{sec:priv_algo}). Finally, the third part is to extend the ``restricted'' estimator to the whole space of inputs while keeping the same privacy and accuracy guarantees by appropriately applying the Extension Lemma , described in Proposition \ref{extension_prop} (see Section \ref{sec:priv_algo_extension}).

\subsection{Proof Sketch of Theorem \ref{thm:main}} \label{sec:sketch}

To provide intuition to the reader, we now briefly describe how one can use the two-stage procedure from Section \ref{sec:approach} to obtain Theorem \ref{thm:main}. We fix an approximate $h$-H\"older query $f(F_q(Y)), Y \sim \mathcal{D}$ and assume for simplicity that it is H\"older w.r.t. $\qd(X,Y),$ that is $A=\Sph.$
\\ 
\paragraph{Obtaining a ``good'' (non-private) estimator} Our first step is to obtain a (non-private) estimate of the query. For that, we sample $n$ independent points $\{X_i\}_{i=1}^n\sim\mathcal{D},$ and define $X$ the uniform empirical measure over $(X_1,\ldots,X_n)$, for which we compute $f(F_q(X)).$ In terms of accuracy, by an appropriate generalization of the arguments in \cite{anderson2020efficiency} to apply for any admissible distribution we have $\qd(X,Y) \leq (\alpha/K)^{1/h}$ for some $n=\tilde{O}(dK^{1/h}/\alpha^{2/h}).$ Notice now that, by admissibility, and the discussion in the ``existence'' paragraph of Section \ref{sec:float_prelim} it can be easily checked that $F_q(Y)$ satisfies the necessary geometric condition; it contains a ball and is contained in a ball, both of ``controlled'' radius. Hence, the approximate H\"olderness implies $\|f(F_q(X))-f(F_q(Y))\|_p \leq K \qd(X,Y)^h=\alpha. $
\\ 
\paragraph{Designing the private estimator on a typical set.}Our goal now turns to design a private yet accurate estimate of $f(F_q(X)).$ To do this, in principle we would desire $f(F_q(X))$ to be Lipschitz (with a ``good'' constant, say $\Lambda_f$) with respect to the Hamming distance on $X.$ Indeed, with such a property the Laplace mechanism \cite{dwork_book} produces an $\eps$-DP estimate of $f(F_q(X))$ with error $\Lambda_f/\eps.$ Unfortunately, such a property cannot exist in general; as mentioned, in many cases of interest, $f$ is only known to be Lipschitz with respect to $\qd$ under certain conditions; for example, when $F_q(X)$ contains a ball and is contained in a ball. So, we must impose some restrictions on the input $X$.

For this reason, we design an appropriate ``typical'' high-probability set $\mathcal{H} \subseteq (\RR^d)^n$, such that $f(F_q(X))$, restricted to $\mathcal{H}$ is Lipschitz, with respect to the Hamming distance on the input $X$. The properties of the typical set will allow that for all $X,Y \in \mathcal{H}$ unless the Hamming distance $d_H(X,Y)$ is ``large'' it holds (a) $\|f(F_q(X))-f(F_q(Y))\|_p \leq K \qd(X,Y)^h$ and (b) $\qd(X,Y) \leq \frac{C}{Ln} d_H(X,Y)$. These results together imply $\|f(F_q(X))-f(F_q(Y))\|_p \leq K C^h/(L^h n^h) d_H(X,Y).$ We then apply a variant of the Laplace mechanism, called the flattened Laplacian mechanism, \cite{borgs2018revealing, tzamos2020optimal} which gives an $\eps$-DP, yet accurate, estimate of $f(F_q(X))$ when $X \in \mathcal{H}.$ The accuracy guarantee of the flattened Laplacian mechanism results from a careful multivariate integral calculation.

Let us now describe the typical set, $\mathcal{H}$. It is built as the intersection of two conditions, each one happening with high-probability. The first condition is that, in every direction, the quantiles are appropriately bounded; a condition which is satisfied by merit of the empirical distribution of $X$ being close to the population distribution $\mathcal{D}$ in the $\qd$ distance. As discussed, this condition enforces property (a) from the paragraph above. The second condition is more intricate, as it requires that, in every direction, the quantile is close, in several scales, to a non-negligible fraction of the points. This reduces the sensitivity of the quantile to the individual sample points and we use it to establish property (b). The proof that the second condition holds with high-probability \emph{for any admissible distribution} is a combination of a net-argument and an appropriate use of the one-dimensional result in \cite[Lemma B.2]{tzamos2020optimal}.
\\ 
\paragraph{Extension Lemma} Finally, a direct application of the Extension Lemma \ref{extension_prop} extends the flattened Laplacian mechanism from the previous paragraph to a $2\eps$-DP estimator on the whole $(\RR^d)^n,$ while remaining the same on inputs from $\mathcal{H}.$ Since $\mathcal{H}$ happens with high probability, the result follows.

\subsection{Admissibility and the Empirical Non-Private Estimator}\label{sec:admiss} 
We now begin the proof. Recall Definition \ref{def:admis}, which introduced the minimal assumption of being an admissible distribution. The first aim of this section is to establish one desirable (yet, non-private) consequence of admissibility; the quantiles of polynomially many samples drawn from an admissible distribution are uniformly close to the quantiles of the distribution (that is, the sample is close to the distribution in the $\qd$ distance). In particular, for any H\"older query $f$, as in \eqref{eq:holder}, simply outputting the value of $f$ on the empirical floating body produces a natural non-private estimator with desirable accuracy using polynomially many samples.

 To formalize and prove this property, we shall require the following lemma from \cite{anderson2020efficiency}.
\begin{lemma}[Lemma 8 in \cite{anderson2020efficiency}] \label{lem:quantilecomp}
	Let $X$ and $Y$ be two random variables with respective CDF $F_Y$ and $F_Y$. Then, for every $q \geq \frac{1}{2}$, if $b:= \sup\limits_t |F_X(x) - F_Y(t)|$ and the following conditions hold, for some $a > 0$:
	\begin{itemize}
		\item $F_X(\Q(X)-a) - F_X(\Q(X)) > b$, and
		\item $F_X(\Q(X)) - F_X(\Q(X)-a) > b$,
	\end{itemize}
	then
	$$|\Q(X)-\Q(Y)| \leq a $$
\end{lemma}
The following is our non-private estimation result which applies to admissible distributions. 

\begin{theorem} \label{thm:FBapprox}
	Let $X$ be a random vector in $\RR^d$ with law in $A_q(R_{\max},R_{\min}, r,L)$. Let $(X_i)_{i=1}^n$ be i.i.d. copies of $X$ and let $Y$ be chosen uniformly from $(X_i)_{i=1}^n$.
	Then, for every $\alpha, \beta > 0$ with $\alpha<\frac{r}{2}$,
	$$\PP\left(\qd(X,Y) \leq \alpha\right) \geq 1-\beta,$$
	provided that,
	$$n \geq \frac{16}{\alpha^2L^2}\left(d\log\left(\frac{16d}{\alpha^2L^2}\right) + \log\left(\frac{4}{\beta}\right)\right).$$
\end{theorem}

\begin{proof}
	We begin by defining the set of hyperplane threshold functions,
	\[
	\mathrm{hyper}_d := \{f: \RR^d \to \{0,1\}| f(x)  = {\bf1}(\langle x, \theta \rangle < t) \text{ for some } \theta \in \Sph, t \in \RR\}.
	\]
	For some $\alpha_1 > 0$, using standard VC arguments, as in \cite[Theorem 7]{anderson2020efficiency}, and the fact that the VC dimension of $\mathrm{hyper}_d$  is $d+1$, we get,
	\begin{equation} \label{eq:vc}
	    \PP\left(\sup\limits_{f \in \mathrm{hyper}_d}|\EE\left[f(X)\right] - \EE\left[f(Y)\right]| \leq \alpha_1 \right) \geq 1-\beta,
	\end{equation}
	whenever,
	\[
	n \geq \frac{16}{\alpha_1^2}\left(d\log\left(\frac{16d}{\alpha_1^2}\right) + \log\left(\frac{4}{\beta}\right)\right).
	\]
	For $\theta \in \mathbb{S}^{d-1}$, let $F_{\langle X, \theta\rangle}$ stand for CDF of $\langle X, \theta\rangle$ (and with a similar notation for $Y$). With this notation, \eqref{eq:vc} may be alternatively written as,
	\begin{equation} \label{eq:goodevent}
		\PP\left(\sup\limits_{\theta,t} |F_{\langle X, \theta\rangle}(t) - F_{\langle Y, \theta\rangle}(t)| \leq \alpha_1\right) \geq 1-\beta.
	\end{equation}
	We now show that the event in \eqref{eq:goodevent} together with Lemma \ref{lem:quantilecomp} and the admissibility of $X$ implies the result. 
	Indeed, assume that $\frac{2\alpha_1}{L} < r$, fix $\theta \in \mathbb{S}^{d-1}$ and denote $X_\theta:= \langle X, \theta \rangle$. In this case, since $X \in A_q(R_{\max},R_{\min}, r,L)$, we have,
	\begin{align*}
		F_{X_\theta}\left(\Q(X_\theta) + \frac{2\alpha_1}{L}\right) - F_{X_\theta}(\Q(X_\theta)) &= \int\limits_{\Q(X_\theta)}^{\Q(X_\theta) + \frac{2\alpha_1}{L}}f_{X_\theta}(t)\mathrm{d}t \geq 2\alpha_1 > \alpha_1,\\
		F_{X_\theta}(\Q(X_\theta)) - F_{X_\theta}\left(\Q(X_\theta) - \frac{2\alpha_1}{L}\right) &= \int\limits_{\Q(X_\theta)-\frac{2\alpha_1}{L}}^{\Q(X_\theta)}f_{X_\theta}(t)\mathrm{d}t \geq 2\alpha_1 > \alpha_1.
	\end{align*}
	Hence, Lemma \ref{lem:quantilecomp} implies,
	$$|\Q(X_\theta) - \Q(Y_\theta)| \leq \frac{2\alpha_1}{L}.$$ 
	for every $\theta \in \mathbb{S}^{d-1}.$
	The proof concludes by choosing $\alpha_1 = \frac{L}{2}\alpha$ for $\alpha<\frac{r}{2}$.
\end{proof}
\subsection{The Typical Set of the Private Estimator} \label{sec:typical}
As mentioned in Section \ref{sec:admiss}, Theorem \ref{thm:FBapprox} implies the success of a non-trivial estimator for any H\"older query: Take a large sample and calculate the query on the empirical floating body. Naturally, such a procedure offers no privacy guarantees.

In order to make this algorithm private, we follow the general approach described in Section \ref{sec:approach} and Section \ref{sec:sketch}. Recall that our first step is to restrict the possible samples into ``typical'' ones. The ``typical''  samples will enjoy two important properties: (a) they are drawn with high-probability over the distribution and  (b) they are not ``too sensitive" to changes in a small number of sample points. These are the two properties we establish in this section. Then, using these properties in Section \ref{sec:priv_algo} we construct of a ``restricted private algorithm'' defined only on the typical set. Finally, with the extension lemma we will produce the final private algorithm defined on every input.
\\ 
\paragraph{Definition of the typical set} We now define a 'typical' subset $\mathcal{H} \subseteq (\mathbb{R}^{d})^n$ of the sample space. Let $\C>1$ be a parameter, and, for a fixed direction $\theta \in \Sph$, define the event:
\[
\mathcal{H}_\C^\theta := \left\{X \in (\mathbb{R}^d)^n: \begin{cases}
	\sum_{i\in [n]}\mathbf{1}\{\langle X_i, \theta \rangle-\Q(\langle X, \theta \rangle )\in[0,\frac{\kappa \C }{Ln}]\}\ge \kappa +1\\
	\sum_{i\in [n]}\mathbf{1}\{\Q(\langle X, \theta \rangle )-\langle X_i, \theta \rangle\in[0,\frac{\kappa \C }{Ln}]\}\ge \kappa +1\\
	\kappa\in\{1,\cdots,\frac{Lr}{2\C}n\}\\
	\Q(\langle X, \theta \rangle) \in [-R_{\max}-\frac{r}{2},R_{\max}+\frac{r}{2}]\\
	\langle X_i, \theta\rangle \leq B, i \in [n]
\end{cases}\right\}.
\]

Now, if $A \subset \Sph$ is the subset of direction we consider, \emph{the typical set} (with respect to $A$) is defined by
\begin{equation} \label{equation:sensitivity-set-median}
	\mathcal{H}(A)=\mathcal{H}_\C(A)=\left(\bigcap\limits_{\theta \in A} \mathcal{H}_\C^\theta\right) \cap \{X \in (\RR^d)^n| F_q(X) \text{ contains a ball of radius }R_{\min}/2\}.
\end{equation} 
We abbreviate $\mathcal{H}(\Sph) = \mathcal{H}$.
\\ 
\paragraph{Intuition behind the definition} We now discuss some intuition. The first two conditions in $\mathcal{H}_\C^\theta$ mean that, when projecting $X$ in direction $\theta$, the quantile has some fraction of points surrounding it in different scales; in each interval $[\Q\left(\langle X, \theta \rangle\right), \Q\left(\langle X, \theta \rangle\right) \pm \frac{\kappa\C}{Ln}]$, there are at least $\kappa + 1$ points, for every $\kappa \in\{1,\cdots,\frac{Lr}{2\C}n\}.$ This is the main property which guarantees the stability of the quantiles under small Hamming distance changes on the input. The third and fourth conditions in $\mathcal{H}_\C^\theta$ as well as the ball containment property in 	$\mathcal{H}(A)$ ensure, respectively, that the quantiles and projections are bounded. 
Note that by Part $4$ in Definition \ref{def:admis} we can, and do, assume
\begin{equation} \label{eq:polyB}
	B =\mathrm{poly}(n,d).
\end{equation}These ``boundedness'' properties crucially implies that the Hausdorff distance between two floating bodies is of comparable size with the ``quantile'' distance $\qd$ between them (e.g. see Lemma \ref{lem:brunel}).
\\ 
\paragraph{High probability guarantees}  
We now show that the typical set is a high-probability event, provided sufficiently many samples are drawn from an admissible distribution.
\begin{lemma}\label{lem:covering}
	Suppose that the sample $X \in (\RR^d)^n$ is drawn from an admissible distribution $\mathcal{D} \in A_q(R_{\max},R_{\min}, r,L)$ and that $\C\leq e^n$ is large enough.
	Then, for any $A \subset \Sph$
	\[
	\mathbb{P}[\mathcal{H}_\C(A)] \geq 1 - \C  e^{-\tilde{\Theta}\left(\C + \min\left(\log(|A|), d\right)\right)} - e^{-\tilde{\Theta}(L^2r^2n + \min\left(\log(|A|), d\right)) },
	\]
	whenever
	\[
	n = \tilde{\Omega}\left(\frac{d}{\min\{R_{\min},r\}^2L^2}\right).
	\]
	In particular, for any $\beta > 0$, we can ensure,
	\[
	\mathbb{P}[\mathcal{H}_\C(A)] \geq 1- \beta,
	\]
	for some
	\[
	n = \tilde{O}\left(\frac{d}{\min\{R_{\min},r\}^2L^2} + \log\left(\frac{1}{\beta}\right)\right) \text{ and } W = \tilde{O}\left(\min\left(\log(|A|), d\right)+\log\left(\frac{1}{\beta}\right)\right).
	\]
\end{lemma} 

The proof, which appears below, is an outcome of combining the one-dimensional Lemma B.3 of \cite{tzamos2020optimal} and an appropriate covering argument.
\\ 
\paragraph{Low Hamming distance sensitivity}
Our next task is to show that typical samples produce empirical quantiles which are not very sensitive to individual sample points. Note that if one changes all $n$ sample points the quantile can take arbitrary values in $[-R_{\min}-r/2,R_{\min}+r/2]$, given that they only need to be realized from input in the typical set. Our next result shows when a fraction of the sample points change, the typical set guarantees the quantiles remain significantly more stable. 
\begin{lemma}\label{lemma:sensitivity-q}
	Let $A \subset \Sph$ and suppose $X, Y \in \mathcal{H}_\C(A)$ with Hamming distance $d_H\left(X,Y\right) \leq \frac{Lr}{2\C}n$. Then,
	$$\qd(X, Y; A) \leq  \frac{2\C}{Ln} d_H\left(X,Y\right).$$
	In particular,
	$$\frac{Ln}{2\C}\min\left\{\qd(X,Y; A),r \right\} \leq d_H(X,Y).$$
\end{lemma}
\subsubsection{Proofs of Lemma \ref{lem:covering} and Lemma \ref{lemma:sensitivity-q}
}\begin{proof}[Proof of Lemma \ref{lem:covering}]
    First, according to Definition \ref{def:admis}, we may assume that when $B$ satisfies \eqref{eq:polyB} with a large enough degree,
    \[
    \PP\left(\max\limits_i \|X\|_2 > B\right) \leq e^{-n}.
    \]
    Moreover, by taking $\alpha = \frac{\min\{R_{\min},r\}}{2}$ and $\beta = e^{-\Theta(n)}$ in Theorem \ref{thm:FBapprox}, we can see that when \[
    n = \tilde{\Omega}\left(\frac{d}{\min\{R_{\min},r\}^2L^2}\right),
    \]
    we have
    \begin{equation} \label{eq:marginalconc}
        \PP\left(\qd(X, \mathcal{D}) > \frac{\min\{R_{\min},r\}}{2}\right)  \leq e^{-\Theta(n)}.
    \end{equation}
    In particular, coupled with Definition \ref{def:qdist}, this shows that there exists $c \in \RR^d$, such that
    \[
    \PP\left(\forall \theta \in \Sph, |\Q(\langle X, \theta) - \langle c, \theta\rangle| \leq \frac{R_{\min}}{2} \text{ and } |\Q(\langle X, \theta)| \leq R_{\max} + \frac{r}{2} \right) \leq e^{-\Theta(n)}.
    \]
    Thus, since our claimed probabilities are of larger order than $e^{-n}$, the rest of the proof is focused on bounding the probabilities of the events
    \begin{align*}
    &\left\{\sum_{i\in [n]}\mathbf{1}\left\{\langle X_i, \theta \rangle-\Q(\langle X, \theta \rangle )\in \left[0,\frac{\kappa \C }{Ln}\right]\right\}\ge \kappa +1\right\},\\
	&\left\{\sum_{i\in [n]}\mathbf{1}\left\{\Q(\langle X, \theta \rangle )-\langle X_i, \theta \rangle\in \left[0,\frac{\kappa \C }{Ln}\right]\right\}\ge \kappa +1\right\},
    \end{align*}
    for $\kappa\in\{1,\cdots,\frac{Lr}{2\C}n\}$.
    
	We now claim that, for a fixed $\theta \in \Sph$,
	\begin{equation} \label{eq:typicalmarginal}
		\PP\left[\mathcal{H}^\theta_\C\right] \geq 1 - \C  e^{-{\Theta}\left(\C\right)} - e^{-\Theta(L^2r^2n)}.
	\end{equation}
	Indeed, this is a consequence of \cite[Lemma B.2]{tzamos2020optimal}. Note that Definition \ref{def:admis} implies that the marginal of $\mathcal{D}$, in direction $\theta$, which we denote as $\mathcal{D}_\theta$, is an admissible distribution around $\Q(\mathcal{D}_\theta)$, in the sense of \cite[Assumption 1.2]{tzamos2020optimal}, and \eqref{eq:typicalmarginal} follows, mutatis-mutandis, from the proof of \cite[Lemma B.2]{tzamos2020optimal}\footnote{We set $T=1$ in \cite[Lemma B.2]{tzamos2020optimal}}.
	
	If $|A| < e^d$, then, with a union-bound
	\begin{align*}
		\PP(\mathcal{H}_\C(A)) &\geq 1 - |A|\left(\C  e^{-\Theta\left(\C\right)} + e^{-\Theta(L^2r^2n)}\right)\\
		& = 1 - \left(\C  e^{-\Theta\left(\C + \log(|A|)\right)} + e^{-{\Theta}(L^2r^2n) + \log(|A|)}\right).
	\end{align*}
	When $|A| \geq e^d$, we will use the inclusion $\mathcal{H}_\C(A)\subset \mathcal{H}_\C$, and prove the result uniformly on the sphere. For this, fix $\gamma:=\frac{\C}{4Bn}$ and let $N_\gamma \subset \Sph$ be an $\gamma$-net. Standard arguments show that one can ensure $|N_\gamma| \leq \left(\frac{3}{\gamma}\right)^d = e^{d\log\left(\frac{3}{\gamma}\right)}$. Denote $\mathcal{H}_\C^{\gamma} := \bigcap\limits_{\theta \in N_{\gamma}} \mathcal{H}_\C$. A union bound over $N_\gamma$ shows,
	\begin{equation} \label{eq:epsnet}
		\PP\left[\mathcal{H}_\C^\gamma\right] \geq 1 - \left(\C  e^{-\Theta\left(\C + d\log\left(\frac{3}{\gamma}\right)\right)} + e^{-\tilde{\Theta}(L^2r^2n) +d\log\left(\frac{3}{\gamma}\right)}\right).
	\end{equation}
	Now, assume $\mathcal{H}_\C^\gamma$ holds, and let $\theta \in \Sph$ with $\theta' \in N_\gamma$ such that, 
	$\|\theta-\theta'\| \leq \gamma$. 
	In this case, since the $\|X_i\|\leq B$, for every $i=1,\dots,n$, 
	$$|\langle X_i, \theta\rangle - \langle X_i, \theta'\rangle| \leq \|X_i\|\|\theta -\theta'\|\leq B\gamma \leq \frac{\C}{4n}.$$ 
	This implies the following bound on the infinity Wasserstein distance: $W_\infty\left(\langle X_i, \theta\rangle, \langle X_i, \theta'\rangle\right) \leq \frac{\C}{4n}$, which in turn implies $|Q_q(\langle X, \theta\rangle) - Q_q(\langle X, \theta'\rangle)| \leq \frac{\C}{4n}$ (the reader is referred to Section 2.3 in \cite{rachev1985monge}, and the discussion following Equation 2.14, for more details). In particular, for every $i$, 
	$$|\langle X_i, \theta\rangle - \Q(\langle X, \theta\rangle)| \leq \frac{\C}{2n} + |\langle X_i, \theta'\rangle - \Q(\langle X, \theta'\rangle)|.$$
	Thus, we have proved the implication  $\mathcal{H}_\C^{\theta}\subset \mathcal{H}_{\frac{\C}{4}}^{\theta'} $. 
	By \eqref{eq:epsnet}, 
	\begin{align*}
		\PP\left(\mathcal{H}_\frac{\C}{4}\right) \geq \PP\left(\mathcal{H}^\gamma_\C\right) &\geq 1 - \left(\C  e^{-\Theta\left(\C + d\log\left(\frac{3}{\gamma}\right)\right)} + e^{-\tilde{\Theta}(L^2r^2n) +d\log\left(\frac{3}{\gamma}\right)}\right)\\
		&\geq 1 - \left(\C e^{-\Theta\left(\C + d\log\left(\frac{12Bn}{\C}\right)\right)} + e^{-\tilde{\Theta}(L^2r^2n) +d\log\left(\frac{12Bn}{\C}\right)}\right).
	\end{align*}
	By \eqref{eq:polyB}, $\log\left(\frac{12Bn}{\C}\right) = O(\log(n))$, hence we subsume it in the $\tilde{\Theta}$ notation and the proof is finished.
\end{proof}

\begin{proof}[Proof of Lemma \ref{lemma:sensitivity-q}] 
	It suffices to show for every $\theta \in A,$ 
	\begin{align}\label{eq:q_gap}\left|\Q(\langle Y,\theta \rangle )-\Q(\langle X,\theta \rangle) \right| \leq \frac{ \C}{2Ln} d_H\left(X,Y\right).
	\end{align}
	We again employ Lemma \ref{lem:quantilecomp}. Let $F_\theta$ stand for the CDF of the empirical distribution of $\langle Y_i,\theta \rangle, i=1,\ldots,n$ and $G_\theta$ for the CDF of the empirical distribution of $\langle X_i,\theta \rangle, i=1,\ldots,n$. It clearly holds that
	$$\|F-G\|_{\infty} \leq \frac{d_H\left(X,Y\right)}{n}.$$ Now, since $X \in \mathcal{H}(A)$ and $d_H\left(X,Y\right) \leq \frac{Lr}{2\C}n$, if $a=\frac{\C}{2Ln}d_H(X,Y)$, it holds that,
	$$G_\theta(\Q(\langle X, \theta\rangle)+a)-G_\theta(\Q(\langle X, \theta\rangle)) > \frac{d_H(X,Y)}{n}.$$
	Then \eqref{eq:q_gap} follows from Lemma \ref{lem:quantilecomp}.
	The second part follows from rearranging the terms. Indeed, note that if $r < \qd(X,Y; A)$, then,
	$$\frac{Ln}{2\C}\min\left\{\qd(X,Y; A),r \right\} = \frac{Lr}{2\C}n \leq d_H(X,Y).$$
\end{proof}
\subsection{The Restricted Private Algorithm: Construction and Analysis} \label{sec:priv_algo}
The aim of this section is to construct a private algorithm and prove that it satisfies the conclusion of our meta-theorem, Theorem \ref{thm:main}. The construction of the algorithm is naturally based on the construction of the typical set in Section \ref{sec:typical}.

\subsubsection{Definition of the Restricted Private Algorithm on the Typical Set}
To prepare the proof, we define a randomized algorithm $\hat{\mathcal{A}}$ on inputs from the typical set $\mathcal{H}=\mathcal{H}_\C(A)$, as defined in Section \ref{sec:typical}.

On input $X \in \mathcal{H}$, the density of the algorithm is given by,
\begin{equation} \label{eq:restricted} 
	f_{\hat{\mathcal{A}}\left(X\right)} (t)=\frac{1}{\hat{Z}_X} \exp\left(-\frac{\eps}{4} \min\left\{ \left(\frac{Ln}{2\C}\right)^h K^{-1} \|t- f(F_q(X))\|_p,\left(\frac{Ln\min\{r,R_{\min}\}}{8\C}\right)^h \right\} \right),
\end{equation}
on the region $\{\|t\|_p \leq 2K(R_{\max}+r/2)\}$. The density is $0$ outside of this region. 
The normalizing constant, $\hat{Z}_X$, is given by, 
\begin{equation} \label{eq:norm_restr} 
	\hat{Z}_X=\int\limits_{\|t\|_p \leq 2K(R_{\max}+\frac{r}{2})} \exp\left(-\frac{\eps}{4} \min\left\{ \left(\frac{Ln}{2\C}\right)^h K^{-1}  \|t- f(F_q(X))\|_p,\left(\frac{Ln\min\{r,R_{\min}\}}{8\C}\right)^h  \right\} \right) \mathrm{d}t.
\end{equation} 
We call this distribution a ``flattened'' Laplacian mechanism, similar to \cite{tzamos2020optimal}.

\subsubsection{Privacy Guarantees of the Restricted Algorithm}
Our first result is to show that $\hat{A}$ is an $\eps/2$-differentially private algorithm we shall require the following lemma.
\begin{lemma}\label{lem:restricted-DP}
	Suppose $n>\frac{2\C}{L}$ and that $
	n = \tilde{\Omega}\left(\frac{d}{\min\{R_{\min},r\}^2L^2}\right)$. The algorithm $\hat{\mathcal{A}}$, defined on $\mathcal{H}(A)$, is $ \frac{\eps}{2}$-differentially private.
\end{lemma}
\begin{proof}
	It will be enough to show that, 
	\begin{equation} \label{eq:private}
		\frac{f_{\hat{\mathcal{A}}(X)}(t)}{f_{\hat{\mathcal{A}}(Y)}(t)} \leq e^{\frac{\eps }{2} d_H(X,Y)}
	\end{equation}
	for every $t \in \RR^M, \|t\|_p \leq 2K(R+r/2)$.
	
	We first observe that, since the event in \eqref{eq:marginalconc} is included in the typical set, for $X, Y \in \mathcal{H}(A)$, $\qd(X,Y;A) \leq R_{\min}/4$.
	Now, by applying the reverse triangle inequality,
	\begin{align*}
		&\frac{\exp\left(-\frac{\eps}{4K}\min\left\{ (\frac{Ln}{2\C})^h K^{-1}\|t- f(F_q(X))\|_p,(\frac{Ln\min\{r,R_{\min}\}}{8\C})^h  \right\}\right)}{\exp\left(-\frac{\eps}{4K}\min\left\{ (\frac{Ln}{2\C})^h K^{-1}  \|t- f(F_q(Y))\|_p,(\frac{Ln\min\{r,R_{\min}\}}{8\C})^h\right\}\right)}\\
		&= \exp\Bigg(-\frac{\eps}{4K}(\min\left\{ (\frac{Ln}{2\C})^h  K^{-1} \|t- f(F_q(X))\|_p,(\frac{Ln\min\{r,R_{\min}\}}{8\C})^h \right\} \\
		&\ \ \ \ \ \ - \min\left\{ (\frac{Ln}{2\C})^h K^{-1} \|t- f(F_q(Y))\|_p,(\frac{Ln\min\{r,R_{\min}\}}{8\C})^h  \right\})\Bigg)\\
		&\leq \exp\left(\frac{\eps}{4K}\min\left\{ (\frac{Ln}{2\C})^h K^{-1} \|f(F_q(X))- f(F_q(Y))\|_p,(\frac{Ln\min\{r,R_{\min}\}}{8\C})^h  \right\}\right)\\
		&\leq \exp\left(\frac{\eps}{4}\min\left\{  (\frac{Ln}{2\C}  \qd(X,Y;A))^h,(\frac{Ln\min\{r,R_{\min}\}}{8\C})^h \right\}\right)\\
			&\leq \exp\left(\frac{\eps}{4}\min\left\{  (\frac{Ln}{2\C}  \qd(X,Y;A))^h,(\frac{Lnr}{8\C})^h \right\}\right)\\
		&\leq \exp\left(\frac{\eps}{4}d_H(X,Y)^h\right) \leq \exp\left(\frac{\eps}{4}d_H(X,Y)\right),
	\end{align*}
	
	where the second inequality is the H\"older property  \eqref{eq:holder} using that $\qd(X,Y;A) \leq R_{\min}/4$, and the second to last inequality is Lemma \ref{lemma:sensitivity-q}. The last inequality is a simple consequence of the fact that the Hamming distance takes non-negative integer values. Since the above inequality holds for every $t \in \RR^M$, we may integrate it to obtain,
	$$\frac{\hat{Z}_X}{\hat{Z}_Y} \leq e^{\frac{\eps}{4}d_H(X,Y)}.$$
	Combining the two estimates gives \eqref{eq:private}.
\end{proof}

\subsubsection{The Accuracy of the Restricted Algorithm}
The analysis of the accuracy of $\hat{\mathcal{A}}$, on the typical set, will be preformed in two steps. First, we will bound the normalizing constant in \eqref{eq:norm_restr} from below, then we shall establish that the integral over \eqref{eq:restricted} is small. We record the following elementary calculation that will facilitate the coming calculations.

\begin{lemma} \label{lem:gammaint}
	Let $k \in \NN,$ and set $g_k(x)=x^k+kx^{k-1}+k(k-1)x^{k-2}+\ldots+k!.$
	Then, for any $0<a<b$ it holds,
	\[ 
	\int\limits_{a}^b t^{k}e^{-t} \mathrm{d}t=g_k(a)e^{-a}-g_k(b)e^{-b}.
	\]
\end{lemma}
\begin{proof}
	We prove the claim by induction. When $k = 0$, $g_0 \equiv 1$, and the base case follows.
	Otherwise, use integration by parts,
	\begin{align*}
	    \int\limits_{a}^b t^{k}e^{-t} \mathrm{d}t&= -t^ke^{-t}\Big\vert_a^b + k\int\limits_{a}^bt^{k-1}e^{-t}\mathrm{d}t = a^ke^{-a} - b^ke^{-b} + kg_{k-1}(a)e^{-a} - kg_{k-1}(b)e^{-b}\\
	    &=(a^k +kg_{k-1}(a))e^{-a} - (b^k +kg_{k-1}(b))e^{-b} = g_k(a)e^{-a} - g_k(b)e^{-b}.
	\end{align*}
	The last identity uses the observation that $g_k(x) = x^k + kg_{k-1}(x)$.
\end{proof}

\begin{lemma}\label{lem:accuracy-typical0}
	Suppose $\C>1$ and let $X \in \mathcal{H}_\C(A)$. Then for any $L,R_{\max}, R_{\min},r>0$ and $\alpha \in (0,\min\{r,R_{\min}\})$, $\beta \in (0,1)$ for some $$n=O\left(\C K^{1/h}\frac{(\log \left(\frac{1}{\beta}\right))^{1/h}+(M \log M)^{1/h}}{(\eps \alpha)^{1/h} L }+\C\frac{M^{1/h} (\log \left( \frac{R_{\max}+r}{\alpha }+1\right))^{1/h}}{(\eps \min\{r,R_{\min}\})^{1/h} L }\right).$$ 
	It holds that \begin{equation*}
		\mathbb{P} [ \|\hat{\mathcal{A}}(X)- f(F_q(X))\|_p \geq \alpha] \leq \beta,
	\end{equation*}
	where the probability is with respect to the randomness of the algorithm $\hat{\mathcal{\mathcal{A}}}$.
\end{lemma}
\begin{proof}
	We prove the claim under the assumption that $K = 1$, in \eqref{eq:holder}, the general case follows by re-scaling by $K$ the target error $\alpha$ and the output of the flattened Laplace mechanism. Moreover, to slightly ease notation we set $r'=\min\{r,R_{\min}\}$, as the variable $R_{\min}$ appears in the definition of the mechanism only together with $r$ via the $\min$ operator. For the final result we simply need to replace $r'$ with $\min\{r,R_{\min}\}.$ We also denote $R_{\max}$ simply by $R.$

	In our calculations we shall use the following change of coordinates: for any function $h:\RR_+ \to \RR$,
	\begin{equation}	\label{eq:polar}
		\int\limits_{\RR^M}h(\|t\|_p)\mathrm{d}t = \omega_{M,p}\int\limits_{0}^\infty x^{M-1}h(x)\mathrm{d}x,
	\end{equation}
	where $\omega_{M,p} > 0$ is some explicit constant (see e.g. \cite[Page 5]{barthe2005probabilistic}). Moreover, by combining the H\"older property of $f$ and that the zero data-set is inside the typical set we have $$\|f(F_q(X))\|_p \leq \qd(X,0;A)\leq (R+r/2)^h \leq R+r/2,$$assuming without loss of generality $R>1.$\\
	\textbf{Step 1}:
	By switching to polar coordinates, as in \eqref{eq:polar}, some elementary algebra and since $r' > r$, we have, 
	\begin{align*}
		\hat{Z}_X&= \int\limits_{\|t\|_p \leq 2(R + r/2)} \exp\left(-\frac{\eps}{4} \min\left\{  \left(\frac{Ln}{2\C}\right)^h \|t- f(F_q(X))\|_{p}, \left(\frac{Lnr'}{8\C}\right)^h  \right\} \right) \mathrm{d}t\\
		& \geq \int\limits_{\|t\|_p \leq R + r/2 } \exp\left(-\frac{\eps}{4} \min\left\{ \left(\frac{Ln}{2\C}\right)^h \|t\|_{p},\left(\frac{Lnr'}{8\C}\right)^h\right\} \right)\mathrm{d}t \\
		& \geq \omega_{M,p}\int\limits_0^{R+r/2} x^{M-1}\exp\left(-\frac{\eps}{4}\min\left\{ \left(\frac{Ln}{2\C}\right)^h x,\left(\frac{Lnr'}{8\C}\right)^h\right\} \right)\mathrm{d}x  \\
		& \geq \omega_{M,p}\int\limits_{0}^{r'/2} x^{M-1}\exp\left(-\frac{\eps}{4} \cdot \left(\frac{Ln}{\C}\right)^h x\right)\mathrm{d}x  \\
		& \geq \omega_{M,p}\left(\frac{4\C^h}{ (n L)^h \eps }\right)^{M-1}\left((M-1)!- g_{M-1}\left(\frac{(n L)^h \eps r' }{8\C^h}\right)\exp\left(-\frac{ (nL)^h \eps r' }{8\C^h}\right)  \right),
	\end{align*} 
	where the last inequality follows from Lemma \ref{lem:gammaint}.
	Now, assuming $n=\omega\left( \frac{\C(M \log M)^{1/h}}{(\eps r')^{1/h} L}\right)$ we have 
	$$g_{M-1}\left(\frac{(n L)^h \eps r' }{8\C^h}\right)\exp\left(-\frac{ (nL)^h \eps r' }{8\C^h}\right)\leq M\left(\frac{(n L)^h \eps r' }{8\C^h}\right)^{M-1}\exp\left(-\frac{(n L)^h \eps r' }{8\C^h}\right)<\frac{(M-1)!}{2},$$ which implies
	\begin{equation}\label{eq:universal_LB}
	    \hat{Z}_X \geq \omega_{M,p}\left(\frac{4\C^h}{ (n L)^h \eps }\right)^{M-1}\frac{(M-1)!}{2}.
	\end{equation} 
	
	\textbf{Step 2}: By applying \eqref{eq:polar}, similarly to Step 1,
	\begin{align*}
		&\mathbb{P} [ \|\hat{\mathcal{A}}(X)-f(F_q(X))\|_p \geq \alpha ] \\
		&\leq \frac{\omega_{M,p}}{\hat{Z}_X} \int\limits_{\alpha}^{3(R+r/2)} x^{M-1}\exp\left(-\frac{\eps}{4} \min\left\{ \left(\frac{Ln}{2\C}\right)^h x,\left(\frac{Lnr'}{8\C}\right)^h\right\}\right)\mathrm{d}x\\
		&\leq \frac{\omega_{M,p}}{\hat{Z}_X} \left( \int_{\alpha}^{+\infty}  x^{M-1}\exp\left(-\frac{\eps }{4} \left(\frac{Ln}{2\C}\right)^h x\right)\mathrm{d}x\ \ + \right.\\&\qquad\qquad\ \left. \int_{r'/2^h}^{3(R+r/2)} x^{M-1}\exp\left(-\frac{\eps }{4}\left(\frac{Lnr'}{8\C}\right)^h\right)\mathrm{d}x \right) \\
		&\leq \frac{\omega_{M,p}}{\hat{Z}_X} \left(\left(\frac{4\C^h}{ (n L)^h \eps }\right)^{M-1}g_{M-1}\left(\frac{\eps}{4} \left(\frac{n L }{2\C}\right)^h \alpha\right)\exp\left(-\frac{\eps}{4} \left(\frac{n L }{2\C}\right)^h \alpha\right)\right.\\&\qquad\qquad\ \ \ + \left.(3(R+ r/2))^M\exp\left(-\frac{\eps }{4}\left(\frac{Lnr'}{8\C}\right)^h\right)\right)
	\end{align*}
	
	Hence we conclude for all $X \in \mathcal{H}_\C(A),$
	\begin{align*}
		&  \mathbb{P} [ \|\hat{\mathcal{A}}(X)-f(F_q(X))\|_p  \geq \alpha]\\
		&\leq 4  \left( \frac{g_{M-1}\left(\frac{\eps}{4} (\frac{n L }{2\C})^h \alpha\right)}{(M-1)!}\exp\left(-\frac{\eps}{4} \left(\frac{n L }{2\C}\right)^h \alpha\right)+\right.\\
		&\left. \qquad\ \ \left(3(R+r/2))\right)^M\left(\frac{(nL)^h\eps }{4\C^h}\right)^{M-1} \exp\left(-\frac{\eps }{4}\left(\frac{Lnr'}{8\C}\right)^h \right)\right).
	\end{align*}
	From this and an elementary asymptotic calculation, we conclude that for $\C>1$, when \[n=\Omega\left(\C\frac{\log \left(\frac{1}{\beta}\right)^{1/h}+(M\log M)^{1/h}}{(\eps \alpha)^{1/h} L }+\C\frac{(M\log (( (\frac{R}{r'}+\frac{r}{r'}+1)M)))^{1/h}+(\log\left(\frac{1}{\beta}\right))^{1/h}}{(\eps r')^{1/h} L }\right),\] 
	it holds, for all $X \in \mathcal{H}$, that  $\mathbb{P} [ \|\hat{\mathcal{A}}(X)-f(F_q(X))\|_p  \geq \alpha] \leq \beta.$ Since $\alpha \leq r'$, the above sample complexity bound simplifies to \[n=\Omega\left(\C\frac{(\log \left(\frac{1}{\beta}\right))^{1/h}+(M \log M)^{1/h}}{(\eps \alpha)^{1/h} L }+\C\frac{M^{1/h} (\log \left( \frac{R+r}{\alpha}+1\right))^{1/h}}{(\eps r')^{1/h} L }\right).\] 
	The proof of the Lemma is complete.
\end{proof}

\subsection{Putting it Together: The Extension Lemma and the Proof of Theorem \texorpdfstring{\ref{thm:main}}{[]}} \label{sec:priv_algo_extension}

In this Section we put everything together and  conclude Theorem \ref{thm:main} from an appropriate use of the Extension Lemma.
\begin{proof}[Proof of Theorem \ref{thm:main}]
We use the Extension Lemma, as in Proposition \ref{extension_prop}, to extend $\hat{\mathcal{A}}$ to the entire space of inputs, $(\mathbb{R}^{d})^n$, endowed with the Hamming distance (for more details on this step we direct the reader to Section \ref{prem:extension}). Call the extension $\mathcal{A}$, and note that $\mathcal{A}$ is $\eps$-differentially private. Indeed, by Lemma \ref{lem:restricted-DP}, $\hat{\mathcal{A}}$ is $\frac{\eps}{2}$-differentially private, and Proposition \ref{extension_prop} implies the required privacy guarantees for $\mathcal{A}$. Moreover, for all $X \in \mathcal{H}_\C(A)$, $\hat{\mathcal{A}}(X)\stackrel{\mathrm{law}}{=} \mathcal{A}(X)$. Thus, we are left with addressing the accuracy of $\mathcal{A}$. 

First, we note that by the assumptions of the theorem 
\[
\C = \Omega\left(\min\{\log(|A|), d \log (B/\C)\} + \log \left(\frac{1}{\beta}\right)\right),
\]
and 
\[
n = \Omega\left(\frac{d}{L^2\min\{R_{\min},r\}^2} + \log \left(\frac{1}{\beta}\right)\right).
\]
Therefore, by Lemma \ref{lem:covering}, with probability $1-\frac{\beta}{2}$, we have $X \in \mathcal{H}(A)$, and so $\mathcal{A}(X),\hat{\mathcal{A}}(X)$ follow the same distribution.

Hence, under the event $X \in \mathcal{H}(A)$, by Lemma \ref{lem:accuracy-typical0}, we have that for some $$n=O\left(\C K^{1/h}\frac{(\log \left(\frac{1}{\beta}\right))^{1/h}+(M \log M)^{1/h}}{(\eps \alpha)^{1/h} L }+\C\frac{M^{1/h} (\log \left( \frac{R_{\max}+r}{\alpha }+1\right))^{1/h}}{(\eps \min\{r,R_{\min}\})^{1/h} L }+\frac{d}{L^2\min\{R_{\min},r\}^2}\right),$$ 
with probability $1-\frac{\beta}{2}$, it holds $$\|\mathcal{A}(X)-f(F_q(X))\|_p \leq \frac{\alpha}{2}.$$

Since $\alpha<KR_{\min}/2$, by Theorem \ref{thm:FBapprox}, there is some
$$n =O\left( \frac{K^{2/h}}{\alpha^{2/h}L^2}\left(d\log\left(\frac{16dK^{2/h}}{\alpha^{2/h}L^2}\right) + \log\left(\frac{4}{\beta}\right)\right)\right),$$
such that, with probability $1-\frac{\beta}{2}$, it holds 
$$\qd(X,\mathcal{D};A) \leq \left(\frac{\alpha}{2K}\right)^{1/h}.$$
Using the triangle inequality and then the H\"older property, \eqref{eq:holder}, we have,
\begin{align*}
   \|\mathcal{A}(X)-f(F_q(\mathcal{D}))\|_p  &\leq \|\mathcal{A}(X)-f(F_q(X))\|_p +  \|f(F_q(X))-f(F_q(\mathcal{D}))\|_p \\
   &\leq \frac{\alpha}{2} + K \qd(X,\mathcal{D};A)^h \leq \alpha.
\end{align*}
A union bound shows that the probability of the event above is at least $1-\beta$. 
The result then follows.
\end{proof}

\section{Private Projections and Sampling} \label{sec:proj and sampling}
In this section we state and prove Corollary \ref{cor:sampling}, which we restate below for convenience.

\begin{corollary}[Restated Corollary \ref{cor:sampling}] \label{cor:sampling_app}
	Let $\mathcal{D} \in A_q(R_{\max}, R_{\min},r,L)$ be an admissible measure on $\RR^d$ and let $U_q$ be a random vector which is uniform on $F_q(\mathcal{D})$. Assume that $F_q(\mathcal{D})$ contains a ball of radius $R_{\min}$ around the Steiner point $S(K)$. Then, there is an $\eps$-differentially private algorithm $\mathcal{A}$ which for input $X=(X_1,\ldots,X_n)$, sampled \emph{i.i.d} from $\mathcal{D}$, satisfies for all $\alpha<\frac{\min\{1,r,R_{\min}\}}{2}$,
	$$\frac{1}{d}W^2_2(\mathcal{A}(X), U_q) \leq \alpha,$$ 
	for some
	$$n=\tilde{O}\left(\frac{\mathrm{poly}(R_{\max}+r +1)}{\mathrm{poly}(R_{\min})}\left(\frac{d^2}{\alpha^{14}L^2}+\frac{d^4}{\eps^2 \alpha^{8}L}+\frac{d^4}{\eps^2 \alpha^2\min\{r,R_{\min}\}^2L}\right)\right).$$
\end{corollary}

\subsection{Proof Sketch}
We start with outlining the steps we follow to establish Corollary \ref{cor:sampling}.
\begin{itemize}
    \item We first employ a non-private algorithm from the sampling literature \cite{lehec2021langevin}, restated below in Theorem \ref{thm:langevinsamp}, which produces an approximate uniform sample from a convex body given access only to a) the Steiner point of the body and b) projection operator onto the body. 
    \item Our next step is to establish that the sampling algorithm is robust to a certain amount of noise. That is, given access to an approximate version of the Steiner point and the projection operator, the non-private algorithm still produces an approximate uniform point from the convex body of interest.
    \item Our final step specializes to privately sampling from the convex body of interest, the floating body of a distribution, where we remind the reader that we are only given samples from the distribution. Using the previous steps, this part is proven by appropriate applications of our meta-theorem, Theorem \ref{thm:main}. We show that Theorem \ref{thm:main} implies that differentially private estimators can achieve the desired approximation guarantees, both when applied to the Steiner point (as proven in Corollary \ref{cor:interior}) and the projection operator (see Corollary \ref{cor:privateproj} below). 
\end{itemize}
\subsection{Background in Wasserstein Distances}\label{sec:wass}

We start with some background material on the $W_p$ distances.

 First, for $p \geq 1$, we define the Wasserstein distance between two random vectors $X, Y\in \RR^d$, as
\[
W_p(X,Y) := \inf\limits_{(X,Y)} \left(\EE\left[\|X -Y\|_2^p\right]\right)^{\frac{1}{p}},
\]
where the infimum is taken over all coupling of $X$ and $Y$; that is, random vectors in $\RR^{2d}$ whose marginal on the first (resp. last) $d$ coordinates has the same law as $X$ (resp. $Y$). 
The Wasserstein distance turns out to be a metric which metrizes weak convergence and convergence of the first $p$ moments, see \cite{rachev1985monge} for further details.
In this work we are mainly interested in the quadratic Wasserstein distance $W_2$. However, note that bounds on $W_2$ gives the same guarantees for $W_p$, when $p \leq 2$. Indeed, by Jensen's inequality,
\begin{equation*} \label{eq:wassjens}
	W_p \leq W_{p'} \text{ whenever } p\leq p'.
\end{equation*} 
\subsection{A Non-private Sampling Algorithm}
If $K\subset \RR^d$ is a convex body, we utilize the following, non-private, sampling algorithm with guarantees in $W_2$. Set $\eta > 0$ and consider the discretized Langevin process:
\begin{equation} \label{eq:langevin}
    X_{t+1} = P_K(X_t + \eta g_t),\ \ \ \ X_0 = S(K),
\end{equation}
where $\{g_t\}_{t \geq 0}$ are \emph{i.i.d.} standard Gaussians, and $S(K)$ is the Steiner point of $K$, as in \eqref{eq:steinerdef}.
The following result holds.
\begin{theorem}[{\cite[Theorem 2]{lehec2021langevin}}] \label{thm:langevinsamp}
	Suppose that $K$ contains a ball of radius $R_{\min}$, centered at $S(K)$, and is contained in a ball of radius $R_{\max}$. Then, if, for some $\alpha > 0$, we take $\eta = \tilde{\Theta}\left(\frac{R_{\min}^2}{(R_{\max}+1)^4}\frac{\alpha^2}{d}\right)$ and $k = \tilde{\Theta}\left((\frac{R_{\max}+1)^6}{R_{\min}^2}\frac{d}{\alpha^2}\right)$, the following bound holds:
	$$\frac{1}{d}W_2^2(X_k, U_K) \leq \alpha,$$
	where $U_K$ is a random vector, uniformly distributed over $K$.
\end{theorem}
\subsection{Noise Robustness of the Non-private Algorithm}

Theorem \ref{thm:langevinsamp} requires \emph{exact access} to the projection operator and to the Steiner point. To allow some uncertainty we now define the notion of a noisy projection oracle.
\begin{definition}[Noisy oracles for $K$]
	We say that the random function $\tilde{P}_K$ is an $(\alpha,\beta, R)$-noisy projection oracle for $K$, if the following two conditions are met, for every $x \in \RR^d$,
	\begin{enumerate}
		\item $\|P_K(x) - \tilde{P}_K(x)\|_2 < R$, almost surely.
		\item $\PP\left(\|P_K(x) - \tilde{P}_K(x)\|_2 <\alpha\right) >1-\beta$.
	\end{enumerate}
By extension we say that the random point $\tilde{S}(K)$ is a $(\alpha,\beta, R)$-noisy oracle for the Steiner point, if it satisfies the same conditions above with respect to the $S(K)$.
\end{definition}
Given noisy oracles for $K$, we can define a noisy version of the Langevin process. In order to take advantage of Corollary \ref{cor:privateproj}, we shall also needs the projected quantities to have bounded norm. Thus, for $k,\alpha > 0$, let $\{\tilde{g}_t\}_{t \geq 0}$ be \emph{i.i.d} random vectors with the law of the standard Gaussian, conditioned to have norm at most $\sqrt{d}\log\left(dk\right)$. We then define the noisy Langevin process as
\begin{equation}\label{eq:noisylangevin}
	\tilde{X}_{t+1} = \tilde{P}_K(\tilde{X}_t + \eta \tilde{g}_t),\ \ \ \ \tilde{X}_0 = \tilde{S}(K)
	.
\end{equation}
We now show that the noisy and noiseless versions cannot differ by much.
\begin{lemma} \label{lem:noisy}
	Fix $k \in \NN$ and suppose that $\tilde{P}_K$ and $\tilde{S}(K)$ are $(\alpha,\beta, R)$-noisy oracles for $P_K$ and $S(K)$, and that $K$ is contained in a ball or radius $R_{\max}$. Then, for every $t \leq k$, there is a coupling of $X_t$ and $\tilde{X}_t$, such that,
	$$\EE\left[\|X_t -\tilde{X}_t\|^2_2\right] \leq (t+1)\left(R^2\beta + \alpha^2 + 4R_{\max}\sqrt{R^2\beta + \alpha^2}\right) +\frac{t}{k}\eta^2.$$
	Consequently,
	$$W_2^2(X_k, \tilde{X}_k) \leq (k+1)\left(R^2\beta + \alpha^2 + 4R_{\max}\sqrt{R^2\beta + \alpha^2}\right) +\eta^2.$$
\end{lemma}
\begin{proof}
	First observe that, if $g$ is a standard Gaussian random vector in $\RR^d$ and $\tilde{g}$ has the law of a standard Gaussian restricted to a ball of radius $\sqrt{d}\log\left(dk\right)$, then,
	\begin{equation} \label{eq:gaussiancoupling}
		W_2^2(g,\tilde{g}) \leq \PP\left(\|g\| > \sqrt{d}\log\left(dk\right)\right)\EE\left[\|g\|^2\right] \leq \frac{1}{k}.
	\end{equation}
	Indeed, if $\gamma$ and $\tilde{\gamma}$ are the respective laws of $g$ and $\tilde{g}$, there is a decomposition,
	$$\gamma = \PP\left(\|g\| \leq \sqrt{d}\log\left(dk\right)\right)\tilde{\gamma} + \PP\left(\|g\| > \sqrt{d}\log\left(dk\right)\right)\gamma',$$
	where $\gamma'$ is $\gamma$ conditioned on being outside the ball of radius $\sqrt{d}\log\left(dk\right)$. This decomposition induces a coupling between $g$ and $\tilde{g}$ which affords the bound in \eqref{eq:gaussiancoupling}. The second inequality in \eqref{eq:gaussiancoupling} follows from $g$ being sub-Gaussian.
	
	We now prove the claim by induction on $t$. The following observation, that arises from the definition of the noisy oracles, will be instrumental:
	One may decompose $\EE\left[\|P_K(\tilde{X}_{t-1} + \eta \tilde{g}_{t-1}) - \tilde{P}_K(\tilde{X}_{t-1} + \eta \tilde{g}_{t-1})\|^2_2\right]$ on the event $\{\|P_K(\tilde{X}_{t-1} + \eta \tilde{g}_{t-1}) - \tilde{P}_K(\tilde{X}_{t-1} + \eta \tilde{g}_{t-1})\|_2 <\alpha\}$ to obtain,
	\begin{equation} \label{eq:boundedexp}
	    \EE\left[\|P_K(\tilde{X}_{t-1} + \eta \tilde{g}_{t-1}) - \tilde{P}_K(\tilde{X}_{t-1} + \eta \tilde{g}_{t-1})\|^2_2\right] \leq R^2\beta + \alpha^2.
	\end{equation}
	The same argument also shows,
	$$\EE\left[\|X_0 - \tilde{X}_0\|^2_2\right] = \EE\left[\|S(K) - \tilde{S}(K)\|^2_2\right] \leq R^2\beta + \alpha^2.$$
	 This establishes the base case of the induction, when $t = 0$.  For $t > 0$, couple the processes $X_t$ and $\tilde{X}_t$, by coupling $g_{t-1}$ and $\tilde{g}_{t-1}$ according to the coupling in \eqref{eq:gaussiancoupling}, and observe
	\begin{align*}
		&\EE\left[\|X_t - \tilde{X}_t\|^2_2\right] = \EE\left[\|P_K(X_{t-1} + \eta g_{t-1}) - \tilde{P}_K(\tilde{X}_{t-1} + \eta \tilde{g}_{t-1})\|^2_2\right]\\
		&=\EE\left[\|P_K(X_{t-1} + \eta g_{t-1}) - P_K(\tilde{X}_{t-1} + \eta \tilde{g}_{t-1}) + P_K(\tilde{X}_{t-1} + \eta \tilde{g}_{t-1}) - \tilde{P}_K(\tilde{X}_{t-1} + \eta \tilde{g}_{t-1})\|^2_2\right].
	\end{align*}
    We also have, from \eqref{eq:boundedexp}, and with Cauchy-Schwartz,
	\begin{align*}
	    \EE&\left[\langle P_K(X_{t-1} + \eta g_{t-1}) - P_K(\tilde{X}_{t-1} + \eta \tilde{g}_{t-1}),P_K(\tilde{X}_{t-1} + \eta \tilde{g}_{t-1}) - \tilde{P}_K(\tilde{X}_{t-1} + \eta \tilde{g}_{t-1})\rangle\right]\\
	    &\leq 2R_{\max}\EE\left[\|P_K(\tilde{X}_{t-1} + \eta \tilde{g}_{t-1}) - \tilde{P}_K(\tilde{X}_{t-1} + \eta \tilde{g}_{t-1})\|_2\right] \leq 2R_{\max}\sqrt{R^2\beta + \alpha^2}
	\end{align*}
	Combining the previous two calculations, we see,
	\begin{align*}
	\EE\left[\|X_t - \tilde{X}_t\|^2_2\right] &\leq\EE\left[\|P_K(X_{t-1} + \eta g_{t-1}) - P_K(\tilde{X}_{t-1}+ \eta \tilde{g}_{t-1})\|^2_2\right] \\
	&\ \ \ + \EE\left[\|P_K(\tilde{X}_{t-1} + \eta \tilde{g}_{t-1}) - \tilde{P}_K(\tilde{X}_{t-1} + \eta \tilde{g}_{t-1})\|^2_2\right]+ 4R_{\max}\sqrt{R^2\beta + \alpha^2}\\
		&\leq R^2\beta + \alpha^2 + 4R_{\max}\sqrt{R^2\beta + \alpha^2} + \EE\left[\|P_K(X_{t-1} + \eta g_{t-1}) - P_K(\tilde{X}_{t-1} + \eta \tilde{g}_{t-1})\|^2_2\right] \\
		&\leq R^2\beta + \alpha^2 + 4R_{\max}\sqrt{R^2\beta + \alpha^2} + \EE\left[\|X_{t-1} + \eta g_{t-1} - (\tilde{X}_{t-1} + \eta \tilde{g}_{t-1})\|^2_2\right] \\
		&= R^2\beta + \alpha^2 + 4R_{\max}\sqrt{R^2\beta + \alpha^2} + \EE\left[\|X_{t-1}  - \tilde{X}_{t-1} \|^2_2\right] + \eta^2\EE\left[\|g_{t-1} - \tilde{g}_{t-1}\|_2^2\right]\\
		&\leq R^2\beta + \alpha^2 + 4R_{\max}\sqrt{R^2\beta + \alpha^2} + t\left(R^2\beta + \alpha^2 + 4R_{\max}\sqrt{R^2\beta + \alpha^2}+\frac{(t-1)\eta^2}{k}\right) +\frac{\eta^2}{k}\\
		&\leq (t+1)\left(R^2\beta + \alpha^2 + 4R_{\max}\sqrt{R^2\beta + \alpha^2}\right) +\frac{t}{k}\eta^2.
	\end{align*}
    The third inequality follows from the fact that, since $K$ is convex, $P_K$ is a contraction, and the penultimate inequality is the induction hypothesis, along with \eqref{eq:gaussiancoupling} and the independence of $g_{k-1}$ and $X_{k-1}$.
\end{proof}
We now identify a regime for the noise parameters $(\tilde{\alpha},\beta, R)$ in which the dynamics in \eqref{eq:noisylangevin} have comparable guarantees to the ones in \eqref{eq:langevin}.

\begin{lemma} \label{lem:privatesampling}
	Suppose that $K$ contains a ball or radius $R_{\min}$, centered at $S(K)$, and is contained in a ball of radius $R_{\max}$. Let $\alpha \in (0,1)$, set $\eta = \tilde{\Theta}\left(\frac{R_{\min}^2}{(R_{\max}+1)^4}\frac{\alpha^2}{d}\right)$, $k = \tilde{\Theta}\left(\frac{(R_{\max}+1)^6}{R_{\min}^2}\frac{d}{\alpha^2}\right)$, and assume that, for some $R > 0$, $\tilde{P}_K$  and $\tilde{S}(K)$ are $(\tilde{\alpha},\beta, R)$-noisy oracles for $P_K$ and $S(k)$. Moreover assume
	$$R^2\beta + \tilde{\alpha}^2 + 4R_{\max}\sqrt{R^2\beta + \tilde{\alpha}^2}\leq \frac{d\alpha}{2k}.$$ Then,
	\[
	\frac{1}{d}W^2_2(\tilde{X}_k, U_K) \leq 9\alpha,
	\]
	where $U_K$ is a random vector, uniformly distributed over $K$.
\end{lemma}
\begin{proof}
	We use the triangle inequality, followed by Theorem \ref{thm:langevinsamp} and Lemma \ref{lem:noisy},
	\begin{align*}
		\frac{1}{\sqrt{d}}W_2(\tilde{X}_k, U_K) &\leq \frac{1}{\sqrt{d}}W_2(\tilde{X}_k, X_k) + \frac{1}{\sqrt{d}}W_2(X_k, U_K)
		\\
		&\leq \sqrt{\frac{k+1}{d}\left(R^2\beta + \tilde{\alpha}^2 + 4R_{\max}\sqrt{R^2\beta + \tilde{\alpha}^2}\right) +\eta^2} + \sqrt{\alpha} \leq 3 \sqrt{\alpha},
	\end{align*}
where we have also used that $\eta^2 \leq \alpha$.
\end{proof}
\subsection{Towards a Private Sampling Algorithm: a Private Noisy Projection Oracle}
Let $K \subset \RR^d$ be a convex body and recall the projection operator, $P_K:\RR^d \to \RR^d$ given by,
$$P_K(x) = \arg\min\limits_y \{y \in K| \|y-x\|_2\}.$$
We first show that when $X$ follows an admissible distribution, one can privately estimate $P_{F_q(X)}(x)$. To this end, we prove Lemma \ref{lem:lip_proj}. The proof uses the following classical result, see \cite[Proposition 5.3]{attouch1993quantitative}.
\begin{proposition}\label{prop:holder}
	Fix $R > 0$ and let $ K_1,K_2 \subset B(0,R)$ be two convex bodies. If $x \in \RR^d$, then $$\|P_{K_1}(x)-P_{K_2}(x)\|_2 \leq 2\sqrt{\left(\|x\|_2+R\right)}\qha(K_1,K_2)^{1/2}.$$
\end{proposition}

Combining the robustness of the projection operator with Lemma \ref{lem:brunel}, we now prove Lemma \ref{lem:lip_proj}.
\begin{proof}[Proof of Lemma \ref{lem:lip_proj}]
	Similar to the proof of Lemma \ref{lem:steiner}, we see that, under the assumption,
	$$\qd(F_q(X),F_q(Y))\leq \frac{R_{\min}}{4},$$
	Lemma \ref{lem:brunel} implies, 
	$$\qha(F_q(X),F_q(Y)) \leq 6\frac{R_{\max} + r}{R_{\min}}\qd(F_q(X),F_q(Y)).$$
	The above bound holds provided $F_q(X)$ contains a ball of radius $R_{\min}/2$, and when $F_q(X),F_q(Y)$ are contained in a ball of radius $R_{\max} +r,$ centered at the origin.
	We invoke Proposition \ref{prop:holder}, according to which
	$$\|P_{F_q(X)}(x)-P_{F_q(Y)}(x)\|_2 \leq 5\sqrt{\left(\|x\|_2+R_{\max} +r\right)\frac{R_{\max}+r}{R_{\min}}}\qd(F_q(X),F_q(Y))^{1/2}$$
	which completes the proof.
\end{proof}
Having established that $P_{(\cdot)}(x)$ is an approximate $\frac{1}{2}$-H\"older function with respect to the convex body, we prove the following corollary.
\begin{corollary} \label{cor:privateproj}
	Let $\mathcal{D} \in A_q(R_{\max}, R_{\min},r,L)$ be an admissible measure on $\RR^d$. Then, for every $x \in \mathbb{R}^d$, there is an $\eps$-differentially private algorithm $\mathcal{A}$ which for input $X=(X_1,\ldots,X_n)$ with i.i.d. entries from $\mathcal{D}$ satisfies for all $\alpha<\frac{\min\{r,R_{\min}\}}{2}$ that
	$$\PP(\|\mathcal{A}(X)-P_{F_q(\mathcal{D})}(x)))\|_2 \leq \alpha) \geq 1-\beta,$$
	for some 
	$$n = \tilde{O}\left(\frac{\left(\|x\|_2+R_{\max} +r\right)^4}{R^2_{\min}}\left(\frac{ \left(d + \log\left(\frac{1}{\beta}\right)\right)}{\alpha^4L^2}+\frac{d^3\log\left(\frac{1}{\beta}\right)^2 + d^2\log \left(\frac{1}{\beta}\right)^3}{\eps^2  \alpha^2L}+\frac{d^3 + d^2\log(\frac{1}{\beta})}{\eps^2 \min\{r,R_{\min}\}^2 L }\right)\right).$$
\end{corollary}
\begin{proof}
	By Lemma \ref{lem:lip_proj} $P_{(\cdot)}(x)$ is an approximate $\frac{1}{2}$-H\"older function with constant, $K = 5\sqrt{\frac{\left(\|x\|_2+R_{\max}+r\right)(R_{\max}+r)}{R_{\min}}}$, with $h=\frac{1}{2}$ and $M = d$.
\end{proof}
\subsection{A Private Sampling Algorithm for the Floating Body} We now focus on building the private sampling algorithm for the floating body, $F_q(\mathcal{D})$, of a distribution $\mathcal{D}$. Recall that we have access to i.i.d. samples drawn from $\mathcal{D}$ and, in light the previous section, we will use the i.i.d. samples to privately approximate the Steiner point and the projection operator for $F_q(\mathcal{D})$.
To be more specific, notice first that the private algorithms defined in Corollaries \ref{cor:interior} and \ref{cor:privateproj} naturally produce noisy oracles for the Steiner point and the projection operators, respectively.  The next Lemma exactly quantifies it in a convenient way for what follows.
\begin{lemma} \label{lem:privateoracle}
	Let $k \geq d$ and substitute $\tilde{\alpha}:=\frac{d\alpha}{32k (R_{\max}+1)}$ for $\alpha$, and $\beta = \frac{d^2\alpha^2}{2^{14}k^2(R_{\max}+1)^2( R_{\max}+r)^2}$, in Corollary \ref{cor:privateproj}. Then, the algorithm promised by the Corollary is $\eps$-differentially private and a $\left(\tilde{\alpha},\beta, 4(R_{\max}+r)\right)$ -noisy projection oracle for $F_q{(\mathcal{D})}$, which uses 
	\[n = \tilde{O}\left(\frac{(\|x\|_2+R_{\max} +r)^4}{R_{\min}^2}\left(\frac{ (R_{\max}+1)^4k^4}{d^3\alpha^4L^2}+\frac{(R_{\max}+1)^2k^2d}{\eps^2  \alpha^2L}+\frac{d^3}{\eps^2 \min\{r,R_{\min}\}^2 L }\right)\right)\]
	samples.
	
	Moreover, with the same parameters and the same bound on the number of samples, the output of the algorithm from Corollary \ref{cor:interior} is a noisy oracle for $S(K)$.
\end{lemma}
\begin{proof}
	Fix $x \in \RR^d$. If $\mathcal{A}(X)$ is the output of the algorithm in Corollary \ref{cor:privateproj}, we have
	\[
	\PP\left(\|P_{F_q{\mathcal{(D)}}}(x) - \mathcal{A}(X)\|_2 \leq \frac{d}{32k (R_{\max}+1)}\alpha\right) \geq 1 -\frac{d^2\alpha^2}{2^{14}k^2(R_{\max}+1)^2(R_{\max} +r)^2}.
	\]
	Moreover, by the construction in \eqref{eq:restricted}, almost surely,
	$$\|P_{F_q{\mathcal{(D)}}}(x) - \mathcal{A}(X)\|_2 \leq  4(R_{\max} + r).$$
	The sample complexity bound follows by appropriately substituting terms.
	
	The proof for the Steiner point is identical, with the bounds obtained in Corollary \ref{cor:interior}.
\end{proof}
Now, recall that Lemma \ref{lem:privatesampling} reveals the level of the noise tolerance of the oracles under which the Langevin still produces an approximately uniform point of the floating body. Our final step is to combine Lemma \ref{lem:privateoracle} with Lemma \ref{lem:privatesampling} to complete the proof of Corollary \ref{cor:sampling}.
\begin{proof}[Putting it all together: Proof of Corollary \ref{cor:sampling}] 
Set $\eta = \tilde{\Theta}\left(\frac{R_{\min}^2}{(R_{\max}+1)^4}\frac{\alpha^2}{d}\right)$, $k = \tilde{\Theta}\left(\frac{(R_{\max}+1)^6}{R_{\min}^2}\frac{d}{\alpha^2}\right)$. We shall invoke Lemma  \ref{lem:privateoracle} to privately compute the initialization and each iteration of the noisy Langevin process, as in \eqref{eq:noisylangevin}.
For $(\tilde{\alpha},\beta, R)$ as defined by Lemma \ref{lem:privateoracle}, we have,
$$R^2\beta + \tilde{\alpha}^2 = 16(R_{\max} + r)^2\frac{d^2\alpha^2}{2^{14}k^2(R_{\max}+1)^2(R_{\max} + r)^2} + \frac{d^2\alpha^2}{1024k^2(R_{\max}+1)^2} = \frac{d^2\alpha^2}{512k^2(R_{\max}+1)^2}.$$
By our choice of $k$, we may freely assume $k \geq d$ and since $\alpha \leq 1$, we have, $\frac{d\alpha}{128k(R_{\max}+1)^2} \leq 1$. Thus,
\begin{align*}
    R^2\beta + \tilde{\alpha}^2 + 4R_{\max}\sqrt{R^2\beta + \tilde{\alpha}^2} \leq \frac{d^2\alpha^2}{512k^2(R_{\max}+1)^2} + \frac{d\alpha}{4k}\leq \frac{d\alpha}{2k},
\end{align*}

and  Lemma \ref{lem:privatesampling} shows that, for the random vector $U_q$,
$$\frac{1}{d}W^2_2(\tilde{X}_k, U_q) \leq 9\alpha.$$
Moreover, note that by definition of the noisy projection oracle, we have, for every $t \leq k$,
\[
\|\tilde{X}_t\|_2\leq R_{\max} + 4(R_{\max} + r) \leq 5(R_{\max} + r).
\]
Hence, recalling that $\tilde{g}_t$ has the law of a standard Gaussian conditioned on the ball of radius $\sqrt{d}\log(dk)$, we have, since $R_{\max} \geq R_{\min}$,
\begin{align*}
	\|\tilde{X}_t +\eta\tilde{g}_t\|_2 &\leq 5(R_{\max} + r) +\eta\sqrt{d}\log(dk) \leq \tilde{O}\left(R_{\max}+r+\log(k)\frac{R_{\min}^2}{(R_{\max}+1)^4}\right)  \\
	& \leq  \tilde{O}\left(R_{\max }+ r+\log\left(\frac{d}{\alpha}\right)\frac{R_{\min}^2}{(R_{\max}+1)^4}\right) = \tilde{O}\left(R_{\max }+ r +1\right).
\end{align*}
Thus, since Lemma \ref{lem:privateoracle} is invoked $k+1$ times, each time with an $x$ satisfying $\|x\|_2 \leq \tilde{O}\left(R_{\max }+r+1\right)$, the sample complexity is
\[n = \tilde{O}\left(\frac{(R_{\max} + r + 1)^8}{R^2_{\min}}\left(\frac{ k^5}{d^3\alpha^4L^2}+\frac{k^3d}{\eps^2  \alpha^2L}+\frac{kd^3}{\eps^2 \min\{r,R_{\min}\}^2 L }\right)\right).\]
Substituting $k$, we get
\[n = \tilde{O}\left(\frac{(R_{\max} +r +1)^{38}}{R_{\min}^{12}}\left(\frac{ d^2}{\alpha^{14}L^2}+\frac{d^4}{\eps^2  \alpha^{8}L}+\frac{d^4}{\eps^2\alpha^2 \min\{r,R_{\min}\}^2 L }\right)\right).\]
\end{proof}
\subsection{Beyond Uniform Sampling} \label{sec:beyonduniformsampling}
Let us note that Theorem \ref{thm:langevinsamp} is a specialized form of a more general result. In fact, \cite[Theorem 2]{lehec2021langevin} offers sampling guarantees for so-called log-concave measures, that is measures with densities of the form $$e^{-\varphi(x)}{\bf 1}_K(x),$$ where $K$ is a convex body, and $\varphi(x)$ is a convex function (the uniform sampling simply sets $\phi$ to be constant). 
A straightforward adaption of our differentially private projection oracle can lead to a differentially private sampler from arbitrary log-concave measures supported on the floating body. The sample complexity would now need to also depend, polynomially, on the Lipschitz constant of $\varphi$; we leave the exact dependence on it for future work.

We chose to state and prove Corollary \ref{thm3_contrib} for the uniform measure on $F_q(\mathcal{D})$. This decision was made both for the sake of simplicity, but also because, arguably, the uniform measure is among the most interesting cases; one may need to ``exclude outliers'' and a uniform sample produces a typical representative from what remains. 

However, let us note one possible application of (low-temperature) sampling from more general measures, which may be of interest in future research in the differential privacy community. It is well known that when $\varphi$ is a convex function, sampling from the measure $e^{-\varphi(x)}{\bf 1}_K(x)$ is intimately connected to \emph{optimization}; that is optimizing $\varphi$ over $K$ (as in \cite{raginsky2017non}). That is, sampling is connected to finding $$\arg\max\limits_{x\in K}\varphi(x).$$Thus, when considering floating bodies, one should be able to privately optimize convex functions over $F_q(\mathcal{D})$, which can be seen as a given data-set ``pruned'' to have no outliers.

\bibliographystyle{alpha}
\bibliography{bib}

\appendix

\newpage
\section{Examples of Admissible Distributions}

\label{sec:appendix_admissible}
In this section we demonstrate one prototypical example of a class of admissible distributions. The example should serve to both show that Definition \ref{def:admis} is not vacuous as well as give the reader some idea of the possible interplay between the different parameters in the definition. We focus on log-concave measures but mention that similar reasoning can be applied to other classes, like $\alpha$-stable laws, see \cite[Section 7]{anderson2020efficiency}.
\\ 
\paragraph{Symmetric log-concave distributions}
A measure $\mathcal{D}$ on $\RR^d$ is said to be log-concave if its has a density $e^{-\varphi(x)}$, such that $\varphi$ is a convex function. Prominent examples of log-concave measures include Gaussians and uniform measures on convex sets.
To enforce a useful normalization we shall consider isotropic measures. These are measures whose expectation is zero, and whose covariance matrix equals the identity. 
Finally, we say that a distribution is symmetric, if when $X \sim \mathcal{D}$, then $X$ and $-X$ have the same law.
Our main result for log-concave measures is as follows. 
\begin{proposition} \label{prop:logconcaves}
    Let $\mathcal{D}$ be a symmetric, log-concave, and isotropic measure on $\RR^d$, and let $q \in (\frac{1}{2},1)$. Then, $\mathcal{D} \in A_q\left(q - \frac{1}{2}, \log\left(\frac{1}{2(1-q)}\right),\frac{1-q}{2},\frac{1-q}{8}\right)$. That is, 
    one can take, 
    \[
    R_{\min} = q - \frac{1}{2}, R_{\max} = \log\left(\frac{1}{2(1-q)}\right), r = \frac{1-q}{2}, L = \frac{1-q}{8}.
    \]
\end{proposition}
The proof of Proposition \ref{prop:logconcaves} is broken down in several lemmas. We begin by showing that in every direction, the $q$ quantile has some mass around it.
\begin{lemma}\label{lem:lb_mass}
	Fix any $q \in (\frac{1}{2},1)$, and $X \sim \mathcal{D}$, a symmetric, log-concave, and isotropic measure on $\RR^d$. Then, if $L=\frac{(1-q)^2}{16}, r=\frac{(1-q)}{2}$, $X$ satisfies that for any direction $\theta \in \Sph$, it holds, for the density $f_\theta$ of $\mathcal{D}_{\theta}:=\langle X,\theta \rangle$, that
	\[
	f_\theta(t) > L
	\]
    when
	\[
	t \in [\Q(\mathcal{D}_{\theta})-r,\Q(\mathcal{D}_{\theta})+r].
	\]
\end{lemma}
\begin{proof}
	For $\theta \in \Sph$, by the Pr\'ekopa-Leindler inequality, $\mathcal{D}_\theta$ is a symmetric, log-concave, and isotropic measure on $\RR$ (see \cite[Theorem 5.1]{lovasz2007geometry}).
	So, by \cite[Lemma 13]{anderson2020efficiency}, $f_\theta(\Q(\mathcal{D}_{\theta})) \geq \frac{1-q}{4}$. Since $f_\theta$ is symmetric and uni-modal it is decreasing on $(0,\infty)$. So, for any $t < \Q(\mathcal{D}_{\theta})$, $f_\theta(t) \geq \frac{1-q}{4}$, as well.
	 $t > \Q(\mathcal{D}_{\theta})$ note that if we take $r <  \frac{1-q}{2}$ then, since $f_\theta \leq 1$, as in \cite[Lemma 5.5]{lovasz2007geometry},
	\begin{align*}
		\int\limits_{-\infty}^{m_q(\mathcal{D}_{\theta})+r}f_\theta(t) \mathrm{d}t = q + \int\limits_{\Q(\mathcal{D}_{\theta})}^{\Q(\mathcal{D}_{\theta})+r}f_\theta(t) \mathrm{d}t\leq q + \int\limits_{\Q(\mathcal{D}_{\theta})}^{\Q(\mathcal{D}_{\theta})+\frac{1-q}2}\mathrm{d}t \leq \frac{q+1}{2}.
	\end{align*}
	Hence, $\Q(\mathcal{D}_{\theta})+r \leq Q_{\frac{q+1}{2}}(\mathcal{D}_{\theta})$ and the same argument as before shows
	$f_\theta(\Q(\mathcal{D}_{\theta})+r) \geq f_\theta(m_{\frac{q+1}{2}}(\mathcal{D}_{\theta})) \geq \frac{1-q}{8}$. In particular, this is true for any $t \in [\Q(\mathcal{D}_\theta), \Q(\mathcal{D}_\theta)+ r].$
\end{proof}
The fact that for log-concave measure the floating body both contains a ball and is contained in a ball of fixed radii was previously proven in \cite{anderson2020efficiency}.
\begin{lemma}[{\cite[Lemma 13]{anderson2020efficiency}}]\label{lem:balls}
	Fix any $q \in (\frac{1}{2},1)$, and $X \sim \mathcal{D}$, a symmetric, log-concave, and isotropic measure on $\RR^d$. Then, for every $\theta \in \Sph$,
	\[
	q - \frac{1}{2}\leq \Q(\langle X,\theta\rangle) \leq 1 + \log\left(\frac{1}{2(1-q)}\right).
	\] 
\end{lemma}
Let us now prove Proposition \ref{prop:logconcaves}.
\begin{proof}[Proof of Proposition \ref{prop:logconcaves}]
    In light of Lemma \ref{lem:lb_mass} and Lemma \ref{lem:balls} it will be enough to show that if $\{X_i\}_{i=1}^n$ are \emph{i.i.d.} as $\mathcal{D}$, then,
    \[
    \PP\left(\max\limits_i\|X_i\|_2 > 10\sqrt{d}n^3\right) \leq e^{-n}.
    \]
    But this is clear, since each $X_i$ has sub-exponential, \cite[Lemma 5.17]{lovasz2007geometry} tails and $\EE[\|X_i\|_2^2] = d$. Indeed, with a union-bound,
\[
    \PP\left(\max\limits_i\|X_i\|_2 > 10\sqrt{d}n^3\right) \leq n\PP\left(\max\limits_i\|X_1\|_2 > 10\sqrt{d}n^3\right) \leq n e^{-10n^2 + 1} \leq  e^{-n}.
    \]
\end{proof}

\section{On the Extension Lemma and the Proof of Theorem \ref{thm:main}}\label{prem:extension}

In this section we provide for the interested reader more details on the Extension Lemma (stated in Proposition \ref{extension_prop}) and how it is used in Section \ref{sec:priv_algo_extension} to construct the final private algorithm to establish Theorem \ref{thm:main}. We repeat here the Extension Lemma for convenience.

\begin{proposition}["The Extension Lemma" Proposition 2.1, \cite{borgs2018private} ] \label{extensionApp} 
Let $\hat{\mathcal{A}}$ be an $\eps$-differentially private algorithm designed for input from $\mathcal{H} \subseteq (\mathbb{R}^d)^n$ with arbitrary output measure space $(\Omega,\mathcal{F})$. Then there exists a randomized algorithm $\mathcal{A}$ defined on the whole input space $(\mathbb{R}^d)^n$ with the same output space which is $2\eps$-differentially private and satisfies that for every $X \in \mathcal{H}$, $\mathcal{A}(X) \overset{d}{=}  \hat{\mathcal{A}}(X)$.
\end{proposition} 
We start with recalling that the input space is $\mathcal{M}=(\mathbb{R}^d)^n$ equipped with the Hamming distance and the output space is $(\RR^M,\|\cdot\|_p)$ equipped with the Lebesgue measure. Furthermore let us assume also that for any $X \in \mathcal{H}=\mathcal{H}(A)$ the randomised restricted algorithm $\hat{\mathcal{A}}(X)$ follows a $\RR^M$-valued continuous distribution with a density $f_{\hat{\mathcal{A}}(X)}$ with respect to the Lebesgue measure, given by  \eqref{eq:restricted}, \eqref{eq:norm_restr}. Applying the Extension Lemma readily gives the $\eps$-differentially private extension $\mathcal{A}$ on input $X \in (\mathbb{R}^d)^n$ which agrees with $\hat{\mathcal{A}}$ when $X \in \mathcal{H}.$
\\ 
\paragraph{Density of the extended algorithm} Now one may wonder if the extended algorithm admits a density and, if so, what does it look like. It turns out that because the restricted algorithm admits a density with respect to the Lebesgue measure, the same holds for the extended algorithm. Moreover, the density has, in fact, a simple-to-state formula. In more detail, an inspection of the proof of the Extension Lemma [Section 4, \cite{borgs2018private}] gives that the ``extended'' $\eps$-differentially private algorithm $\mathcal{A}$ admits a density, on input $X \in (\mathbb{R}^d)^n$, given by
\begin{equation}\label{alg:extended}
f_{\mathcal{A}(X)}(\omega) =\frac{1}{Z_X} \inf_{X' \in \mathcal{H}} \left[ \exp\left(\frac{\eps}{4} d_H(X,X')\right) f_{\hat{\mathcal{A}}(X')}(\omega) \right], \omega \in \mathbb{R}^M\end{equation} where \begin{equation*}Z_X:=\int_{\mathbb{R}^M} \inf_{X' \in \mathcal{H}} \left[ \left(\frac{\eps}{4} d_H(X,X')\right) f_{\hat{\mathcal{A}}(X)'}(\omega) \right] \mathrm{d} \omega.
\end{equation*}
For reasons of completeness we state here (and prove in the following section) the corollary of the Extension Lemma that establishes that in our setting the algorithm $\mathcal{A}$ satisfies the desired properties of the Extension Lemma. 
\begin{proposition}\label{prop_replica}
Under the above assumptions, the algorithm $\mathcal{A}$ with density given in \eqref{alg:extended} is $\eps$-differentially private and for every $X' \in \mathcal{H}$, $\mathcal{A}(X')\overset{d}{=}\hat{\mathcal{A}}(X')$.
\end{proposition}

As a technical remark note that, in order for the density in equation \eqref{alg:extended} to be well-defined we require that, for every $X \in (\mathbb{R}^d)^n$ the ``unnormalized'' density function \begin{equation} \label{eq:Gfunction} G_X(\omega):=  \inf_{X' \in \mathcal{H}} \left[ \exp\left(\frac{\eps}{2} d_H(X,X')\right) f_{\hat{\mathcal{A}}(X')}(\omega) \right],  \omega  \in \mathbb{R}^M  \end{equation} 
is integrable. 
This condition follows from the following Lemma, which establishes - among other properties - that $G_X$ is a continuous function almost everywhere and has a finite integral.
\begin{lemma}\label{well_posed}
Suppose the above assumptions hold and fix $X \in \mathcal{H}(A)$. Then, \begin{itemize} 
\item $G_X(\omega)=0$ for all $\omega \not \in \mathcal{I}=\{ \|\omega\|_p \leq 2K(R+r/2)\}.$ 
\item $G_X$ is $\mathcal{R}$-Lipschitz on $\mathcal{I}$ with a universal (that is, independent of the value of $X \in \mathcal{H}$) Lipschitz constant $\mathcal{R}<\infty.$
\end{itemize}Furthermore, it holds $0 \leq \int_{\omega \in \mathbb{R}} G_X(\omega) \mathrm{d}\omega \leq 1.$
\end{lemma}

Finally, notice that the infimum over $X' \in \mathcal{H}$ in \eqref{alg:extended} is the main reason the extension lemma comes with no explicit termination time guarantees; the termination time largely depends on the ``nature" of $\mathcal{H}.$ See the appendix of \cite{tzamos2020optimal} for a further discussion on this point.
	
\subsection{Proof of Proposition \ref{prop_replica}}

We start with using \cite[Lemma 4.1.]{borgs2018private} applied to our setting. This gives the following Lemma.
\begin{lemma}\label{lemma:Radon-Nikodyn-equivalence}
Let $\mathcal{A}'$ be a real-valued randomized algorithm designed for input from $\mathcal{H}' \subseteq (\mathbb{R}^d)^n$. Suppose that for any $X \in \mathcal{H}'$, $\mathcal{A}'(X)$ admits a density function with respect to the Lebesgue measure $f_{\mathcal{A}'(X)}$. Then the following are equivalent \begin{itemize} \item[(1)] $\mathcal{A}'$ is $\eps$-differentially private on $\mathcal{H}$; \item[(2)] For any $X,X' \in \mathcal{H}$ \begin{equation}\label{prime}  f_{\mathcal{A}'(X)}(\omega) \leq e^{\eps d_H(X,X') } f_{\mathcal{A}'(X')}(\omega), \end{equation} almost surely with respect to the Lebesgue measure. 
\end{itemize}
\end{lemma}

\begin{proof}[Proof of Proposition \ref{prop_replica}]

We first prove that $\mathcal{A}$ is $\eps$-differentially private over all pairs of input from $(\mathbb{R}^d)^n$. Using Lemma \ref{lemma:Radon-Nikodyn-equivalence} it suffices to prove that for any $X_1,X_2 \in \mathcal{H}$,
 \begin{align*}
f_{\mathcal{A}(X_1)}(\omega) \leq \exp \left( \eps d_H(X_1,X_2)\right) f_{\mathcal{A}(X_2)}(\omega),
\end{align*} almost surely with respect to the Lebesgue measure. We establish it in particular for every $\omega \in \mathbb{R}^M$. Notice that if $\omega \not \in \mathcal{I},$ both sides are zero from Lemma \ref{well_posed}. Hence let us assume $\omega \in \mathcal{I}$. Let $X_1,X_2 \in \mathbb{R}^n$. Using triangle inequality we obtain for every $\omega \in \mathcal{I}$, \begin{align*}\inf_{X' \in \mathcal{H}} \left[ \exp \left( \frac{\eps}{2} d_H(X_1,X')\right) f_{\hat{\mathcal{A}}(X')}(\omega) \right] & \leq \inf_{X' \in \mathcal{H}} \left[ \exp\left(\frac{\eps}{2} \left[d_H(X_1,X_2)+d_H(X_2,X')\right] \right) f_{\hat{\mathcal{A}}(X')}(\omega) \right]\\
&=\exp \left( \frac{\eps}{2} d_H(X_1,X_2)\right) \inf_{X' \in \mathcal{H}} \left[ \exp \left(\frac{\eps}{2} d_H(X,X')\right) f_{\hat{\mathcal{A}}(X')}(\omega) \right],
\end{align*}which implies that for any $X_1,X_2 \in \mathcal{M}$, 
\begin{align*}
Z_{X_1}&=\int_{\Omega} \inf_{X' \in \mathcal{H}} \left[ \exp\left(\frac{\eps}{2} d(X_1,X')\right) f_{\hat{\mathcal{A}}(X')}(\omega) \right] \mathrm{d} \omega \\
&\leq \exp\left(\frac{\eps}{2} d(X_1,X_2) \right)\int_{\Omega} \inf_{X' \in \mathcal{H}} \left[ \exp \left( \frac{\eps}{2} d(X_2,X') \right) f_{\hat{\mathcal{A}}(X')}(\omega) \right] \mathrm{d}  \omega\\
&=\exp \left( \frac{\eps}{2}d(X_1,X_2) \right) Z_{X_2}. 
\end{align*}Therefore using the above two inequalities we obtain that for any $X_1,X_2 \in \mathbb{R}^n$ and $\omega \in \mathcal{I}$,
 \begin{align*}
f_{\mathcal{A}(X_1)}(\omega) &=\frac{1}{Z_{X_1}} \inf_{X' \in \mathcal{H}} \left[ \exp \left( \frac{\eps}{2} d_H(X_1,X') \right) f_{\hat{\mathcal{A}}(X')}(\omega) \right] \\
&\leq \frac{1}{\exp\left(-\frac{\eps}{2} d_H(X_2,X_1)\right)Z_{X_2}}\exp\left(\frac{\eps}{2}d_H(X_1,X_2)\right) \inf_{X' \in \mathcal{H}}  \left[ \exp \left( \frac{\eps}{2} d(X_2,X') \right) f_{\hat{\mathcal{A}}(X')}(\omega) \right] \\
&=\exp \left( \frac{\eps}{2} d(X_1,X_2)\right) \frac{1}{Z_{X_2}} \inf_{X' \in \mathcal{H}}  \left[ \exp \left( \frac{\eps}{2}d_H(X_2,X') \right) f_{\hat{\mathcal{A}}(X')}(\omega) \right]\\
&=\exp \left( \eps d_H(X_1,X_2)\right) f_{\mathcal{A}(X_2)}(\omega),
\end{align*}as we wanted. 

Now we prove that for every $X \in \mathcal{H}$, $\mathcal{A}(X) \overset{d}{=}  \hat{\mathcal{A}}(X)$. Consider an arbitrary $X \in \mathcal{H}$. We know that $\hat{\mathcal{A}}$ is $\frac{\eps}{2}$-differentially private which based on Lemma \ref{lemma:Radon-Nikodyn-equivalence} implies that for any $X,X' \in \mathcal{H}$ \begin{equation}  
f_{\hat{\mathcal{A}}(X)}(\omega) \leq \exp \left(\frac{\eps}{2} d_H(X,X') \right) f_{\hat{\mathcal{A}}(X')}(\omega),   
\end{equation} 
almost surely with respect to the Lebesgue measure. Observing that the above inequality holds almost surely as equality if $X'=X$ we obtain that for any $X \in \mathcal{H}$ it holds 
$$f_{\hat{\mathcal{A}}(X)}(\omega)=\inf_{X' \in \mathcal{H}}  \left[ \exp \left( \frac{\eps}{2} d_H(X,X') \right) f_{\hat{\mathcal{A}}(X')}(\omega) \right],$$
almost surely with respect to the Lebesgue measure. Using that $f_{\hat{\mathcal{A}}(X)}$ is a probability density function  we conclude that in this case 
$$Z_{X}=\int  f_{\hat{\mathcal{A}}(X)}(\omega) \mathrm{d}\omega =1.$$ 
Therefore $$f_{\hat{\mathcal{A}}(X)}(\omega)=\frac{1}{Z_X}\inf_{X' \in \mathcal{H}}  \left[ \exp \left( \eps d_H(X,X') \right) f_{\hat{\mathcal{A}}(X')}(\omega) \right],$$ almost surely with respect to the Lebesgue measure and hence $$ f_{\hat{\mathcal{A}}(X)}(\omega)=f_{\mathcal{A}(X)}(\omega),$$ almost surely with respect to the Lebesgue measure. This suffices to conclude that $\hat{\mathcal{A}}(X) \overset{d}{=}  \mathcal{A}(X)$ as needed. The proof of  Proposition \ref{prop_replica} is complete.
\end{proof}

\subsection{Proof of Lemma \ref{well_posed}}

\begin{proof}[Proof of Lemma \ref{well_posed}] First notice that if $\omega \not \in \mathcal{I},$ from \eqref{eq:restricted} for any $X' \in \mathcal{H},$ $f_{\hat{\mathcal{A}}(X')}(\omega)=0$. Therefore indeed $$0 \leq f_{\mathcal{A}(X)}(\omega) \leq \exp\left(\frac{\eps}{2} d_H(X,X')\right) f_{\hat{\mathcal{A}}(X')}(\omega)=0.$$

We prove now that for all $X' \in \mathcal{H},$  the function $\exp\left(\frac{\eps}{2} d_H(X,X')\right) f_{\hat{\mathcal{A}}(X')}(\omega)$ is $\mathcal{R}$-Lipschitz on $\mathcal{I}$. The claim then follows by the elementary real analysis fact that the pointwise infimum over an arbitrary family of $\mathcal{R}$-Lipschitz functions is an $\mathcal{R}$-Lipschitz function.

Now recall that for all $a,b>0$ by elementary calculus, $|e^{-a}-e^{-b}| \leq |a-b|.$ Hence, for fixed $X'\in \mathcal{H},$ using the definition of the density in equation \eqref{eq:restricted}, we have for any $\omega, \omega' \in \mathcal{I},$ \begin{align*}
& |f_{\hat{\mathcal{A}}(X')}(\omega)-f_{\hat{\mathcal{A}}(X')}(\omega')|  \\
& \leq \frac{\eps \left(\frac{Ln}{2\C}\right)^h}{4\hat{Z}_{X'}} |\min\left\{  K^{-1} \|\omega- f(F_q(X'))\|_p,\min\{r,R_{\min}\} \right\}-\min\left\{  K^{-1} \|\omega'- f(F_q(X'))\|_p,\min\{r,R_{\min}\} \right\}|
\end{align*}Now combining with triangle inequality we conclude 
\begin{align*}
|f_{\hat{\mathcal{A}}(X')}(\omega)-f_{\hat{\mathcal{A}}(X')}(\omega')| \leq  \frac{\eps \left(\frac{Ln}{2\C}\right)^h}{4\hat{Z}_{X'}} \min\left\{  K^{-1} \|\omega- \omega'\|_p,\min\{r,R_{\min}\} \right\} \leq \frac{\eps \left(\frac{Ln}{2\C}\right)^h}{4K\hat{Z}_{X'}}\|\omega- \omega'\|_p.
\end{align*}
Following the proof of the accuracy guarantee (Lemma \ref{lem:accuracy-typical0}) we have for all $X' \in \mathcal{H},$
\begin{equation}\label{eq:universal_LB_2}
	    \hat{Z}_{X'} \geq \hat{Z}:=\omega_{M,p}\left(\frac{4\C^h}{ (n L)^h \eps }\right)^{M-1}\frac{(M-1)!}{2}.
	\end{equation} Hence it holds that
	\begin{align*}
|f_{\hat{\mathcal{A}}(X')}(\omega)-f_{\hat{\mathcal{A}}(X')}(\omega')| \leq  \frac{\eps \left(\frac{Ln}{2\C}\right)^h}{4K\hat{Z}}\|\omega- \omega'\|_p.
\end{align*}Finally,  $\exp\left(\frac{\eps}{2} d_H(X,X')\right)$ is a constant independent of $\omega$ with $\exp\left(\frac{\eps}{2} d_H(X,X')\right) \leq \exp(\frac{\eps n}{2} ) $.

Combining the above, $\exp\left(\frac{\eps}{2} d_H(X,X')\right) f_{\hat{\mathcal{A}}(X')}(\omega)$ is $\mathcal{R}= \exp(\frac{\eps n}{2})\frac{\eps \left(\frac{Ln}{2\C}\right)^h}{4K\hat{Z}}$-Lipschitz and notice that $\mathcal{R}$ is independent of $X'$. The proof of the Lipschitz continuity is complete.

The final part follows from the fact that $G$ is non-negative by definition. Moreover, again by definition, for arbitrary fixed $X' \in \mathcal{H},$ $f_{\hat{\mathcal{A}}(X')}$
integrates to one and upper bounds, pointwise, the function $G_X.$
\end{proof}
\end{document}